\documentclass[a4paper,fleqn]{cas-sc}

\usepackage[authoryear,longnamesfirst]{natbib}
\usepackage{caption}
\usepackage{subcaption}
\usepackage{amsmath}
\usepackage{amsthm}
\usepackage{float}
\usepackage{placeins}
\usepackage{bm}
\usepackage{color,soul}
\usepackage{hyperref} 
\usepackage{xcolor}
\DeclareUnicodeCharacter{00A0}{ }
\DeclareMathOperator*{\argmax}{arg\,max}

\makeatletter
\AtBeginDocument{\let\hl\@firstofone}
\makeatother
\def\tsc#1{\csdef{#1}{\textsc{\lowercase{#1}}\xspace}}
\tsc{WGM}
\tsc{QE}
\tsc{EP}
\tsc{PMS}
\tsc{BEC}
\tsc{DE}

\newtheorem{theorem}{Theorem}
\newtheorem{lemma}[theorem]{Lemma}
\newtheorem{proposition}[theorem]{Proposition}
\newtheorem{definition}{Definition}
\newproof{pf}{Proof}
\newproof{pot}{Proof of Theorem \ref{thm}}

\usepackage[colorinlistoftodos]{todonotes}

\begin{document}



\title [mode = title]{A Novel Max Pressure Algorithm Based on Traffic Delay}                      



%
\author[1]{Hao Liu}[orcid=0000-0002-5319-9514]

\cormark[1]


\ead{hfl5376@psu.edu}

\address[1]{Department of Civil and Environmental Engineering, The Pennsylvania State University, University Park, PA, USA}
\author[1]{Vikash V. Gayah}
\ead{vvg104@psu.edu}
\begin{abstract}
This paper considers a novel travel-delay-based Max Pressure algorithm for control of arbitrary transportation networks with signalized intersections. The traditional number-of-vehicle-based Max Pressure (Original-MP) algorithm has received tremendous attention recently due to its ease of implementation and scalability to large network scenarios. The Original-MP algorithm also has a desirable property  called maximum stability, which means a demand scenario can be accommodated by this method as long as it can be accommodated by any existing control policy. However, implementation of the Original-MP Max Pressure algorithm may be difficult in practice as estimation of queue lengths at intersections requires significant measurement infrastructure. The Original-MP framework also uses a point queue model to represent the vehicle transition between links, which does not consider the position of the vehicles, even though this may  significantly impact the control performance. In addition, intersection approaches with  low travel demand can incur arbitrarily large delays due to having short queues. The proposed travel-delay based Max Pressure model overcomes these drawbacks while inheriting the maximum stability feature of the Original-MP. It also is shown to outperform several benchmark Max Pressure variants in a battery of simulation tests. Lastly, the proposed algorithm can be implemented in a Connected Vehicle (CV) environment, in which a subset of vehicles serve as mobile probes. The results show the proposed travel-delay-based model generates lower delay with non-full penetration rate than the benchmark models with full penetration rate under certain traffic conditions.

\end{abstract}



\maketitle
\section{Introduction}
Traffic congestion has been a steadily growing problem for the urban network users and  is expected to only worsen over time if current trends continue \citep{jia2015survey}. Traffic signal control, which aims to optimize traffic conditions by coordinating the conflict vehicle movements at signalized intersections, serves as a cost effective strategy to alleviate congestion in urban networks. Although it has drawn a considerable amount of research efforts, traffic signal control is a complex problem and there are still many questions that need to be answered. One particularly  challenging topic is the network-wide traffic signal control, which increases in difficultly due to the  interdependence between adjacent  intersections within the network. However, centralized control methods that include direct  coordination between all network intersections are not scalable owing to the drastic increase in computational burden with the size of the network.

To address the scalability issue, decentralized models have been proposed, which only require local traffic states and optimize signal timing at each individual intersection separately. The Max pressure (MP) policy, also known as backpressure policy, is one more recently decentralized model that has  been investigated vastly since its proposal for traffic signal control during the last decade. The MP policy was initially presented in \citep{tassiulas1990stability} for the routing and scheduling of packet transmission in a wireless network. \citep{varaiya2013max}  proposed an MP control model in the context of a network of signalized intersections. This algorithm uses point queues to model the vehicle evolution on a link and uses queue length, which is the number of vehicles in the point queues, at each intersection approach to determine what signal phase to implement at that intersection. Specifically, the algorithm requires the queue lengths in adjacent links and the turning ratios between links at the intersection. Using this information, the weight for each potential movement is calculated as the difference between its own queue length and average queue length in all downstream links. The phase with maximum pressure, which is equal to the sum of weights of the served movements, is selected for implementation. In addition to decentralization, the favorable properties of this algorithm are twofold: 1) no need for prior knowledge on demand; and, 2) maximization of throughput, i.e., the MP control algorithm can accommodate any demand as long as there exists a control policy can accommodate it, which cannot be ensured by some classical control algorithms such as SCATS  \citep{sims1979sydney}, SCOOT \citep{robertson1991optimizing}, OPAC \citep{gartner2001implementation}, etc. 

Despite these desirable properties, the MP algorithm has assumptions that deviate from real traffic network, which limits its applicability in practical. Thanks to its decentralization and simplicity, different MP variations \citep{levin2020max, le2015decentralized, xiao2014pressure, gregoire2014capacity,  mercader2020max, wu2017delay, dixit2020simple} have been proposed to address certain limitations. However, none of these algorithms guarantee the optimization of certain metrics, such as delay, throughput, queue length, etc. The control performance in terms of such metrics highly depends on the selection of parameters used to calculate the weight and pressure. Most of the MP variants use queue lengths at the instant when signal is updated. However, since the vehicles are not distinguished by their moving status, i.e., the moving vehicles are treated the same as the stopped vehicles, which might result in poor performance under specific situations. In addition, even if it uses real queue length to address this issue, the difficulty of queue measurements may also limit the applicability of queue-based MP variants in practice. Instead, non-queue-based alternatives have been recently proposed. For example, \citep{mercader2020max} proposed a travel time-based MP (TT-MP) algorithm that uses travel time along approach links to calculate each movement's pressure. 

Minimization of delay, which is the additional travel time experienced by a driver compared to the `ideal' travel time at free flow and without obstructions, is one of the most commonly used objectives in traffic control problems.\citep{ji2012delay} \hl{proposed a delay-based MP (D-MP) schedule scheme for wireless networks to address the excessive delay incurred by flows due to the lack of subsequent packet arrivals. To the best of our knowledge, only two delay-based MP (D-MP) algorithms} \citep{wu2017delay, dixit2020simple} \hl{for traffic signal control have been proposed in the literature. However, both models have shortcomings, which will be discussed later in Section} \ref{sec:lr}, \hl{that may limit their practicability. This paper proposes a novel D-MP that can overcome these drawbacks.} The proposed model is analytically proved to own the key property of the MP in \citep{varaiya2013max}, i.e., maximum stability. In addition, the performance is explored through microscopic simulation, and the simulation results show that the proposed model outperforms three benchmark models in commonly studied metrics such as delay and throughput under various traffic conditions.

The rest of this paper is organized as follows. Section \ref{sec:lr} reviews MP policies that have been proposed. Section \ref{sec:q_tt} briefly introduces the basic mechanism of a MP model and two benchmark models. Section \ref{sec:d_mp} shows the development of the proposed D-MP and the proof of the maximum stability. The numerical case study results using SUMO are presented in Section \ref{sc:cs}, and we conclude in Section \ref{sc:cc}.

\section{Literature review}\label{sec:lr}
The MP algorithm was initially implemented in traffic signal control by \citep{varaiya2013max}. It is referred to as Original-MP in the remainder of this paper. In the Original-MP model, traffic operation is modeled using point queue models and queue capacity is assumed to be infinite, i.e., vehicles are never blocked by downstream queue spillover. The pressure is defined only based on the difference in queue lengths between upstream and downstream links; note that this does not consider priorities according to the absolute queue lengths, which can be disadvantageous. For example, intuitively, it should be more beneficial to serve a movement with a long upstream queue even if it does not have the largest weight because this can reduce the risk of spillover. These are the major bottlenecks that limit its practicability in real traffic signal systems. Thanks to its decentralization and simplicity, different variations of MP have been proposed to address these drawbacks. Table \ref{tab:lit} summarizes the MP variants in literature.
\begin{table}
	\caption{A summary of MP variants}
	\centering
	\begin{tabular}{p{3cm} p{3cm} p{3cm} p{5cm}}
		\Xhline{3\arrayrulewidth}
		\textbf{Study name} & \textbf{Pressure measure} & \textbf{Signal Update} & \textbf{Special feature/notes} \\
		\Xhline{3\arrayrulewidth}
		Max pressure \citep{varaiya2013max, lioris2016adaptive}	& Instant queue length & Every time step & Infinite link capacity;\newline maximum stability;\newline not work conservative\\
		Modified max pressure \citep{kouvelas2014maximum} & Instant queue length & Every cycle & Finite link capacity;\newline normalized queue length\\ 
		Pressure releasing policy \citep{xiao2014pressure} & Instant queue length & Every time step & Consider ingress queues in addition to the queues at the optimized intersection;\newline able to handle downstream spillover; \newline maximum stability within a reduced region\\
		MP with fixed cycle length \citep{le2015decentralized} & Instant queue length & Every cycle & Maximum stability;\newline fixed phase sequence; \newline able to adjust the green time difference\\
		Delay-based MP \citep{wu2017delay} & Sojourn time of the Head-Of-Line vehicle & Every time step & Maximum stability;\newline better fairness than Original-MP;\newline works for isolated intersections\\
		Position weighted backpressure \citep{li2019position} & Instant vehicle spatial distribution & Every time step & weight depends on vehicle locations;\newline Work conservative;\newline strong stability\\
        Travel time based MP \citep{mercader2020max} & Normalized average travel time during a cycle & Every cycle & Normalized travel time; \newline work conservative\\
        MP with maximum cycle length \citep{levin2020max} & Instant queue length & Every time step & Cycle length is upper bounded;\newline fixed phase sequence;\newline rolling horizon scheme;\newline maximum stability within a reduced region\\
        \hl{Delay-based MP} \citep{dixit2020simple} & Travel delay & Every cycle & Cycle length is fixed;\newline fixed phase sequence\\
		\Xhline{4\arrayrulewidth}
	\end{tabular}\label{tab:lit}
\end{table}
\citep{xiao2014pressure} proposed a so-called "pressure releasing policy" (PRP) to consider queue capacities, and it was proved that the stability is maintained with a reduction on the throughput region. In addition to the adjacent queue lengths at an intersection, the PRP also takes the ingress queues (queues in the entry links) into consideration for the definition of phase pressure. Instead using the absolute queue lengths, \citep{kouvelas2014maximum, wei2019presslight} used the ratio between the absolute queue length and the queue capacity as the pressure to consider the effect of link length. \citep{gregoire2014capacity} proposed a new MP policy in which the queue length is normalized by a convex function of queue capacity. The study claims that new vehicles joining a queue can be more problematic as the queue grows. The strict convexity of the normalization is capable of increasing marginal pressure due to an additional vehicle with the rising of queue. The model was proved to be work conservative, which can reduce the probability of blocking. It was also shown that the capacity-aware max pressure model outperforms the Original-MP mainly under heavy traffic loads. Although the queue capacity is considered, the work conservation is only valid for certain circumstances in terms of the spatial distribution of vehicles due to the use of point queue model. To address this issue, \citep{li2019position} proposed a position-weighted backpressure (PWBP) model  that is able to capture the spatial distribution of vehicles and spill-back dynamics. 

Signal updating frequency plays another important role on the control policy. In the Original-MP, pre-defined discrete time periods are used and the signal phase with the maximal pressure is activated at the beginning of every period. \citep{lioris2016adaptive} numerically tested the effect of five update frequencies on control performance in terms of sum of queue lengths in the network and found increased frequency of control updates leads to diminished network queues. However, these tests only considered a certain demand scenario and the time period selection was thus not comprehensive. In contrast, we will numerically show that there exists a threshold for the update frequency above which the control performance can be deteriorated. Under this updating fashion, the signal timing including both phase sequence and duration from this policy could be undesirable since the disorder of phases brings frustrations to drivers and the waiting time for the approach with small demand can be arbitrarily long. As an alternative, another type of policy was proposed, cycle-based MP algorithms \citep{le2015decentralized, kouvelas2014maximum}, that allocates effective green time to each phase based on the ratios of pressures at the beginning of each cycle. However, the cycle length in all these models are fixed. \citep{levin2020max} proposed a model with fixed phase sequence and maximal cycle length. The cycle lengths between intersections can be different, but the effect of this discrepancy on control performance is still open to answer.

All above models use queue lengths to calculate weight and pressure. However, as mentioned before, the queue length in the Original-MP is equal to the number of vehicles rather than the stopped vehicles. In addition, queue length might be more difficult than other metrics to be obtained, as argued in \citep{mercader2020max}. \hl{Instead, other measures might be more readily available via floating vehicles or mobile vehicle probes. For example,} \citep{dixit2020simple} \hl{acquired high-quality real-time traffic delay at a cheaper cost compared to physical sensors, which are required for queue measurement. The delay data was used as the input of their model of which the feasibility has been demonstrated at real intersections.} In addition, with the rapid development and popularization of the connected vehicle (CV) technology, metrics like travel time and delay can be estimated with an acceptable accuracy from the information sent from the CVs even in a relatively low penetration rate. However, an accurate estimation on queue length requires a high coverage rate of detectors or infrastructure, which is usually not available in reality. Even in a CV environment, although the queue length can be estimated given the penetration rate, the error could be considerable especially when the penetration rate is low. Therefore, it is necessary to explore the MP variants based on different metrics. Several variants proposed non-queue length metrics that can be used to calculate the pressure. \citep{mercader2020max} proposed a travel time-based MP (TT-MP) that is inherently capacity-aware, meaning the marginal increase in travel time with a long queue is higher than a short queue. It was shown that the proposed model has a lower unstable probability than the queue based model in \citep{kouvelas2014maximum}. \citep{le2015decentralized} pointed out that one drawback of the Original-MP is that the delay of one movement with low travel demand can be arbitrarily large. Delay-based models can resolve this issue since the movement with the largest delay will be served right away. Moreover, delay is also inherently capacity-aware, i.e., the marginal delay increases with queue length. Therefore, compared to the Original-MP which only relies on the difference in queue lengths,  delay-based models can reduce the risk of spillover. The reason is that delays increases faster as queue length grows. Therefore, the delay based MP tends to activate a phase with a long upstream queue even it does not own the largest queue length-based weight. In addition, traffic delay might be the mostly used metric to indicate traffic operational performance, so the delay-based MP is a promising topic to be studied. Very few delay-based MP variants exist in the literature. A delay-based MP using Head-Of-Line (HOL) delay information for pressure was proposed in \citep{wu2017delay} to ensure the fairness with respect to mean delay between different movements; however, the algorithm proposed in that study  only works for isolate intersections. \citep{dixit2020simple} \hl{proposed a parsimonious cycle-based D-MP that uses delay information provided by crowdsourced data to calculate the optimal green time allocation.  The main drawback of this model is that it can only relates long-term average queue lengths and delays, but it cannot describe the relationship between queue lengths and delay at an arbitrary time. As a result, this model cannot verify the stable region or prove the maximum stability property, which is one of biggest strengths of the MP algorithm. In addition, the performance of their model was not clear compared to other MP variants under the conditions with or without a full knowledge about the traffic state, e.g., a fully or partially connected vehicle environment.} Overall, the study of   delay-based MP variants is inadequate. 

\hl{Triggered by all the motivation above, this paper proposes a novel delay-based MP (D-MP) that works for networks and inherits the maximum stability property.} In our model, vehicles on a link are divided into two groups: moving vehicles and stopped vehicles. The traffic operation is modeled using a store-and-forward model, and total delay is defined as the sum of stopped time over all vehicles on a link. We demonstrated that the proposed model inherits the most desirable property of the Original-MP: maximum stability. In addition, the simulation results show the proposed model outperforms three benchmark MP variants (Original-MP, TT-MP and a version of the Original-MP that considers only stopped vehicles, termed the H-MP in this paper) under various traffic conditions. Moreover, we also studied the control performance of the proposed model in a connected vehicle (CV) environment. Before we move to the next section, we clarify and emphasize the terms for measurements used in this paper. Before moving to the next section, we need to clarify the definition of `queue' in this paper. Due to the usage of point queue models, the original-MP algorithm and some MP variants mentioned in this section use the term `queue' to represent the total number of vehicles. However, the proposed model in this paper divides vehicles into two groups: moving vehicles and stopped vehicles, so keeping this definition may lead to confusion. To avoid the confusion, in the remaining of this paper, `queue' is used to indicate the stopped vehicles (i.e,. those that have traveled through the link and have joined the queue at the signal), and `the number of vehicles' is used to represent the total vehicles.

\section{Original-MP and total travel time-based MP models}\label{sec:q_tt} 
In this section, we briefly introduce the Original-MP model proposed in \citep{varaiya2013max} and the TT-MP proposed in \citep{mercader2020max} as the first two benchmark models in this paper. First, we define some terms that are used in the remainder of this paper. The directional road segment between two adjacent intersections is called a link. A pair of links $(l,m)$ that vehicles are allowed to travel from link $l$ and $m$ by the signal setting is called a movement. The term of phase is used to indicate a set of movements that are served by the same signal duration. In general, the basic idea of a MP model is that to first calculate the weight of a movement as the difference of a selected metric, such as the number of vehicles, between this movement and the average value in its downstream movements. Then, the pressure of a phase is the sum of the weights over all movements served by this phase. Next, the phase with the maximum  pressure is activated to allow flows with the highest overall pressure to proceed. Note that the time step for updating traffic signals, denoted by $T$, and the time step for updating traffic dynamics  can be different. For example, many MP models use free flow travel time on a link as the time step to update signal phases, but the time step for traffic simulation is usually shorter. 

Let $x(l,m)$ denote the metric used to calculate the weight, $C(l,m)$ be the mean value of saturation flow of movement $(l,m)$, $H(l,m)$ indicate the turning ratio from link $l$ to $m$, $\mathbb{O}_l$ be the set of downstream links of link $l$ and $\mathbb{S}_j$ represent the set of movements served by phase $j$. The general forms of the weight of a movement $w(l,m)$ and the pressure of phase $j$ $p(S_j)$ at time $t$ can be expressed as Equations \eqref{eq:weight1} and \eqref{eq:pressure1},
\begin{equation}\label{eq:weight1}
    w(l,m)(t) = x(l,m)(t)-\sum_{n\in \mathbb{O}_m}x(m,n)H(m,n)
\end{equation}
\begin{equation}\label{eq:pressure1}
    p(S_j)(t)=\sum_{(l,m)\in \mathbb{S}_j}C(l,m)(t)w(l,m)(t)
\end{equation}
Then, for each intersection, the MP algorithm selects the phase with the largest pressure to activate, shown as
\begin{equation}\label{eq:mp1}
    S^*(t)=\argmax_jp(S_j)(t)
\end{equation}

In the Original-MP, the traffic flow on a link is modeled using point queue model in which the lengths of both the link and queue are ignored and the queue capacity of a link is assumed to be infinite. The current number of vehicles at time $t$ is used in the weight calculation. As shown in Equations \eqref{eq:weight1}-\eqref{eq:mp1}, the model requires the mean values of saturation flow of each movement and turning ratios at each intersection, both of which can be allowed to vary depending on traffic operation. The desirable properties of the Original-MP includes: it is a decentralized control framework so the solution can be obtained very fast; it does not require the knowledge about demand; it can serve all demand if the demand can be served by any signal plan, which is called maximum stability. The third property is the most important and will be explained in detail in the next section.

As is shown in Equation \eqref{eq:weight1}, the weight is a linear function of the number of vehicles. However, in reality, it might be more reasonable to activate a movement with more vehicles even if the corresponding phase does not have the largest weight to avoid queue spillover. To this end, \citep{mercader2020max} proposed a cycle-based TT-MP in which the travel time in a cycle is used to define the weight. The authors show that the travel time is a convex function of maximum of queue length in a signal cycle, and this convexity ensures the marginal weight increases as the queue grows; thus, the  proposed model gives more priority to congested links. Although TT-MP does not guarantee the maximum stability, the authors proved the model is work conservative, i.e., the active phase serves at least one vehicle unless there is no vehicle in the upstream links or all downstream links are blocked. Since the proposed model in this paper is time step-based, i.e., the signal timing is updated every time step, we use the time step-based TT-MP as another benchmark model, in which the total travel time incurred in time interval $[(n-1)T,nT]$ is used to calculate the weight for $n$th time step. Proposition \ref{pro:tt} shows that total travel time in a period is equivalent to the average number of vehicles. Therefore, while Original-MP uses the number of vehicles at a point in time to define the weight and pressure, the essence of TT-MP is to employ the average value of the same metric. Unlike the instant number of vehicles which only provides the traffic state at a time point, total travel time conveys more detailed information during the whole time interval, and it is expected to lead better performance, which will be demonstrated in the case study.

\begin{proposition}\label{pro:tt}
Total travel time in a certain period is equivalent to the average number of vehicles. 
\end{proposition}
\begin{proof}
Without loss of generality, assume the time step for traffic operation is 1 second. Then, the total travel time in each time step is equal to the number of vehicles. Therefore, the total travel time in any period is equal/equivalent to the total/average number of vehicles.
\end{proof}

\section{Proposed delay-based MP model}\label{sec:d_mp}
Since the Original-MP always activates the phase with the largest difference in the number of vehicles between the upstream and the downstream links, one drawback is that the delay of one movement with low travel demand can be arbitrarily large \citep{le2015decentralized}. In addition, the minimization of traffic delay, which is the additional travel time above the travel time a vehicle would experience at free flow, is the most common objective in traffic signal control problems. To this end, we propose a delay-based MP (D-MP) model, in which total traffic delay incurred in $[(n-1)T, nT]$ is used to define the weight in a MP algorithm. Note, although delay is used to compute the weights and pressure, it does not guarantee the minimization of delay. In fact, none of the proposed MP models in the literature ensures the optimization of any metric such as delay, throughput, etc. However, the proposed model is demonstrated to outperform both the Original-MP and TT-MP under various traffic conditions in the case study section. More importantly, the improvement in the control performance does not sacrifice the key merits of the original MP including: decentralization, no need for the knowledge of mean demand and the maximum stability property.

The notations used in the following are summarized in Table \ref{tab:para}.
 \begin{table}[t]
	\caption{Notations}
	\centering
	\begin{tabular}{p{3cm} p{10cm}}
		\Xhline{3\arrayrulewidth}
		\textbf{Notation} & \textbf{Definition} \\
		\Xhline{3\arrayrulewidth}
		\multicolumn{2}{c}{\textbf{Sets}}\\
		$\mathbb{L}$ & set of links \\
		$\mathbb{L}_e$ & set of entry links \\
		$\mathbb{L}_{in}$ & set of links except for internal links\\
		$\mathbb{I}_l$ & set of incoming links of link $l$\\
		$\mathbb{O}_l$ & set of outgoing links of link $l$\\
		$\mathbb{D}_f$ & set of feasible demand\\
		$\mathbb{V}_s(l,m)(t)$ & set of vehicles of movement (l,m) at time t\\
		$\mathbb{V}_s(l,m)(t)$ & set of stopped vehicles of movement (l,m) at time t\\
		$\mathbb{V}_m(l,m)(t)$ & set of moving vehicles of movement (l,m) at time t\\
		\Xhline{4\arrayrulewidth}
		\multicolumn{2}{c}{\textbf{Parameters}}\\
		$x(l,m)(t)$ & number of vehicles in movement (l,m) at time $t$\\
		$x_s(l,m)(t)$ & number of stopped vehicles in movement (l,m) at time $t$\\
		$x_m(l,m)(t)$ & number of moving vehicles in movement (l,m) at time $t$\\
		\hl{$x_{m,s}(l,m)(t)$} & number of moving vehicles in movement (l,m) at time $t$ that join the stopped vehicles in the next step\\
		$T$ & time step size for signal updating\\
		\hl{$x_{total}^{dis}(l,m)(t)$} & number of discharged vehicles in movement (l,m) between [t, t+1]\\
		\hl{$x_{s}^{dis}(l,m)(t)$} & number of stopped vehicles that can be discharged in movement (l,m) between [t, t+1]\\
		$H(l,m)$ & turning ratio from link l to link m\\
		$d(l,m)$ & demand for movement (l,m) if $l\in \mathbb{L}_e$\\
		$f_l$ & average traffic volume for link $l$\\
		$C(l,m)(t)$ & saturation flow for movement (l,m) at time $t$\\
		$c(l,m)$ & mean of the saturation flow for movement (l,m)\\
		$\bar{C}(l,m)$ & maximum of the saturation flow for movement (l,m)\\
		$e_i(t)$ & travel distance for vehicle $i$ in the time step $t$\\
		$\bm{e}(l,m)(i)$ & travel distance for movement (l,m) in the $i$th time interval\\
		$\bm{tt}(l,m)(i)$ & total travel time for movement (l,m) in the $i$th time interval\\
		$b_i(t)$ & traffic delay for vehicle $i$ in the time step $t$\\
		$B(l,m)(t)$ & traffic delay for movement (l,m) in the time step $t$\\
		$\bm{B}(l,m)(i)$ & traffic delay for movement (l,m) in the $i$th time interval\\
		$S_{ij}(t)$ & 1 if phase $j$ at intersection $i$ is activated, 0 otherwise.\\
		$S(l,m)(t)$ & 1 if phase serving movement (l,m) is activated, 0 otherwise.\\
		$v_f$ & free flow speed\\
		$\Delta t$ & time step size for traffic operation\\
		$w(l,m)(t)$ & delay based weight of movement (l,m) at time $t$\\
		$w'(l,m)(t)$ & number of vehicle-based weight of movement (l,m) at time $t$\\
		$p(S_i)$ & pressure at intersection $i$\\
		\Xhline{4\arrayrulewidth}
		\multicolumn{2}{c}{\textbf{Functions}}\\
		$\bm{V}(l,m)(t)$ & distribution of vehicles in link (l,m) at time $t$\\
		$\Psi_1$ & a function mapping the distribution of vehicles to the number of vehicles joining the stopped vehicles\\
		$\Psi_2$ & a function mapping the distribution of vehicles to the maximal number of departing vehicles\\
		\Xhline{4\arrayrulewidth}
	\end{tabular}\label{tab:para}
\end{table}

\subsection{Traffic evolution}
As the same as Original-MP and TT-MP, the traffic operation is modeled using a store-and-forward model. The discharge of a point queue is expressed as
\begin{equation}\label{eq:discharge}
    x_{dis}(l,m)(t+T)=\min\{C(l,m)(t+T)S(l,m)(t)T, x(l,m)(t)\}
\end{equation}
Equation \eqref{eq:discharge} shows the discharge in one time step is equal to the minimum of the number of vehicles that can discharge at saturation during a single time step and the number of vehicles on the link at the beginning of the time step. Consequently, although the travel time is not considered in a point queue, the reasonable time step size in traffic environment is restricted to be equal to the free flow travel time traversing the link, like in \citep{levin2020max}. If $T$ is longer than the free flow travel time $t_f$,  new incoming vehicles from the upstream link of link $l$ could potentially join the stopped vehicles and discharge in the same step, so the discharge is underestimated by the term of $x(l,m)(t)$ in the bracket. On the other hand, if $T$ is shorter than $t_f$,  vehicles near the upstream end of link $l$ are not able to reach the stop line in the current step, so the term of $x(l,m)(t)$ will overestimate the discharge.

To replace the number of vehicles with traffic delay, we categorized  vehicles on a link as being either  stopped vehicles or  moving vehicles; see  Figure \ref{fig:traffic_evo} in which the open-end box at the downstream end of a link indicates the queue and the number inside is the number of stopped vehicles at time $t_0$, blue vehicles are moving at step $t_0$ and stopped at $t_0+1$, and green vehicles are moving in both time steps. Furthermore, the following assumptions are made:\begin{itemize} \item Like the Original-MP, the queue capacity of each link is assumed to be infinite, i.e., the queue length is ignored, as shown in Figure \ref{fig:traffic_evo}. This assumption ignores the impact of downstream traffic state on the discharge from upstream links and allows the number of stopped vehicles to increase arbitrarily; \item At any time $t$, there exists at most one group of stopped vehicles on a link and the group is at the stop line at the intersection, as shown in Figure \ref{fig:traffic_evo}; \item Moving vehicles always travel at free flow speed before they join the stopped vehicles, i.e., moving vehicles do not incur delay; \item The number of moving vehicles on a link is upper bounded by a constant. The reason is that a moving vehicle will either join the stopped vehicles or leave the link within a constant time (free flow travel time on the link). On the contrary, the number of stopped vehicles is allowed to rise  infinitely. This assumption plays a key role in the maximum stability, and we prove this assumption in Proposition \ref{pro:upperbound_moving}.  \end{itemize} Next, we show the traffic flow models for both groups. Figure \ref{fig:traffic_evo} provides the traffic evolution on an activated movement between two consecutive steps for three different scenarios and is used to help explain the models.\\
The store-and-forward model for stopped vehicles can be expressed as,
\begin{equation}\label{eq:stop_evo}
    x_s(l,m)(t+1)=x_s(l,m)(t)+x_{m,s}(l,m)(t)-x_s^{dis}(l,m)(t), \forall l\in \mathbb{L}, m\in \mathbb{O}_l
\end{equation}
where
\begin{equation}
	\begin{split}
	    x_{m,s}(l,m)(t) & = \Psi_1(\boldsymbol{V}(l,m)(t))\\
	    x_s^{dis}(l,m)(t) & = \min\{C(l,m)(t+1)S(l,m)(t+1), x_s(l,m)(t)+\Psi_1(\boldsymbol{V}(l,m)(t))\}
	\end{split}
\end{equation}
\hl{$x_{m,s}(l,m)(t)$ is the number of moving vehicles at $t$ that join the stopped vehicles at $t+1$; $x_s^{dis}(l,m)(t)$ is the number of stopped vehicles that discharge in the time step.} $\boldsymbol{V}(l,m)(t)$ is a time-dependent variable indicating the distribution of vehicles in link $(l,m)$ at time $t$. $\Psi_1$ is a function mapping the distribution of vehicles to $x_{m,s}(l,m)(t)$, as shown by blue vehicles in Figure \ref{fig:traffic_evo}. Note given the assumption that moving vehicles always travel at free flow speed before they stopped, the form of $\Psi_1(.)$ is only related to the free flow speed (although the form can be complex). For simplicity, we assume the free flow speed is a constant, so the form of $\Psi_1(.)$ is time-invariant. This source does not only depend on the number of vehicles in both groups, it also depends on the distribution of the vehicles in the link. For example, the top and middle lanes have the same number of vehicles in each group, but two vehicles on the top lane will join the stopped vehicles, while only one vehicle will in the middle lane because other vehicles are too far away to join the stopped vehicles in the next step. \hl{In addition, since this function denotes the number of vehicles joining the stopped vehicles in one time step, we assume $\Psi_1$ is finite.} $x_s^{dis}(l,m)(t)$ is equal to the minimum of the number of vehicles that can discharge at saturation during a time step and the number of stopped vehicles present. Note this term only counts the departed vehicles that have stopped before leaving the link. The total served vehicles can exceed this term if moving vehicles traverse the link without stops, as shown in the bottom lane in Figure \ref{fig:traffic_evo}, and that portion of outflow will be considered in the evolution of moving vehicles. Equation \eqref{eq:move_entry_evo} shows the evolution of moving vehicles in an entry link,
\begin{equation}\label{eq:move_entry_evo}
    x_m(l,m)(t+1)=x_m(l,m)(t)-x_{m,s}(l,m)(t)-[x_{total}^{dis}(l,m)(t) - x_s^{dis}(l,m)(t)]+d(l,m)(t+1), \forall l\in \mathbb{L}_e,  \mathbb{O}_l
\end{equation}
where
\begin{equation*}
x_{total}^{dis}(l,m)(t) = \min\{C(l,m)(t+1)S(l,m)(t+1), \Psi_2(\boldsymbol{V}(l,m)(t))\}
\end{equation*}
\hl{$x_{total}^{dis}(l,m)(t)$ is the total number of departed vehicles.} $\Psi_2$ is a function mapping the vehicle distribution to the maximum number of vehicles that can leave the link. Therefore, the difference in the square bracket indicates the number of moving vehicles that depart the link without stops. \hl{Since $x_{total}^{dis}(l,m)(t)$ is always upper bounded by $C(l,m)(t+1)S(l,m)(t+1)$, we do not need $\Psi_2$ to be finite.} In Figure \ref{fig:traffic_evo}, this term  is zero for the top and middle lanes, but it is equal to one for the bottom lane since one vehicle is able to  leave the link without stopping. The last term is the demand for the entry link at time $t+1$.

Similarly, the moving vehicles for the internal links can be expressed as,
\begin{equation}\label{eq:move_in_evo}
    \begin{split}
        x_m(l,m)(t+1)=&x_m(l,m)(t)-x_{m,s}(l,m)(t)-[x_{total}^{dis}(l,m)(t) - x_s^{dis}(l,m)(t)] \\
        & +\sum_{k\in\mathbb{I}_l}x_{total}^{dis}(l,m)(t)H(k,l)(t+1), \forall l\in \mathbb{L}_{in},  \mathbb{O}_l
    \end{split}
\end{equation}
The last term indicates for internal links, the source for the moving vehicles is the outflow from upstream links. 
\begin{proposition}\label{pro:upperbound_moving}
The number of moving vehicles on a link is upper bounded by a constant if a moving vehicle will either join the stopped vehicles or leave the link after a constant time (free flow travel time on the link). 
\end{proposition}
\begin{proof}
For concision, Equation \eqref{eq:move_in_evo} can be reformed as
\begin{equation}
    x_m(l,m)(t+1)=x_m(l,m)(t)-x_m^-(l,m)(t+1)+\sum_{k\in\mathbb{I}_l}x_{total}^{dis}(l,m)(t)H(k,l)(t+1), \forall l\in \mathbb{L}_{in},  \mathbb{O}_l
\end{equation}
where $x_m^-(l,m)(t+1)$ is the number of moving vehicles that join the stopped vehicles or depart link $(l,m)$, i.e., the sum of the second and the third terms in Equation \eqref{eq:move_in_evo}. Then, for any $t_1>1$, we have
\begin{equation}\label{eq:moving2}
    x_m(l,m)(t+t_1)=x_m(l,m)(t)-\sum_{t'=1}^{t_1}x_m^-(l,m)(t+t')+\sum_{t'=1}^{t_1}\sum_{k\in\mathbb{I}_l}x_{total}^{dis}(l,m)(t+t'-1)H(k,l)(t+t'), \forall l\in \mathbb{L}_{in},  \mathbb{O}_l
\end{equation}
Since we assume the moving vehicles will join the stopped vehicles or leave the link after the free flow travel time (this could be realized by functions $\Psi_1$ and $\Psi_2$), $t_f$, we have
\begin{equation}\label{eq:moving3}
    x_m^-(l,m)(t+t_f) \ge  \sum_{k\in\mathbb{I}_l}x_{total}^{dis}(l,m)(t-1)H(k,l)(t), \forall l\in \mathbb{L}_{in},  \mathbb{O}_l, t
\end{equation}
Therefore, combining Equation \eqref{eq:moving2} and Equation \eqref{eq:moving3}, for any $t_1>t_f$, we have
\begin{equation}
    x_m(l,m)(t_1)<x_m(l,m)(0)+\sum_{t'=t_1-v_f+1}^{t_1}\sum_{k\in\mathbb{I}_l}x_{total}^{dis}(l,m)(t+t'-1)H(k,l)(t+t'), \forall l\in \mathbb{L}_{in},  \mathbb{O}_l
\end{equation}
From Equation \eqref{eq:move_entry_evo}, we know $x_{total}^{dis}(l,m)(t+t'-1)H(k,l)(t+t')$ is upper bounded by a constant (saturation flow). Therefore, $x_m(l,m)(t)$ is upper bounded by a constant for any time $t$ as well.

This can be proved for the entry links in a similar way.
\end{proof}

Note that the evolution of the aggregated number of vehicles, shown as Equation \eqref{eq:agg}, is the same as the flow model in the Original-MP.

\begin{equation}\label{eq:agg}
    x(l,m)(t)=x_s(l,m)(t)+x_m(l,m)(t)
\end{equation}
\begin{figure}
    \centering
    \includegraphics[width=6in]{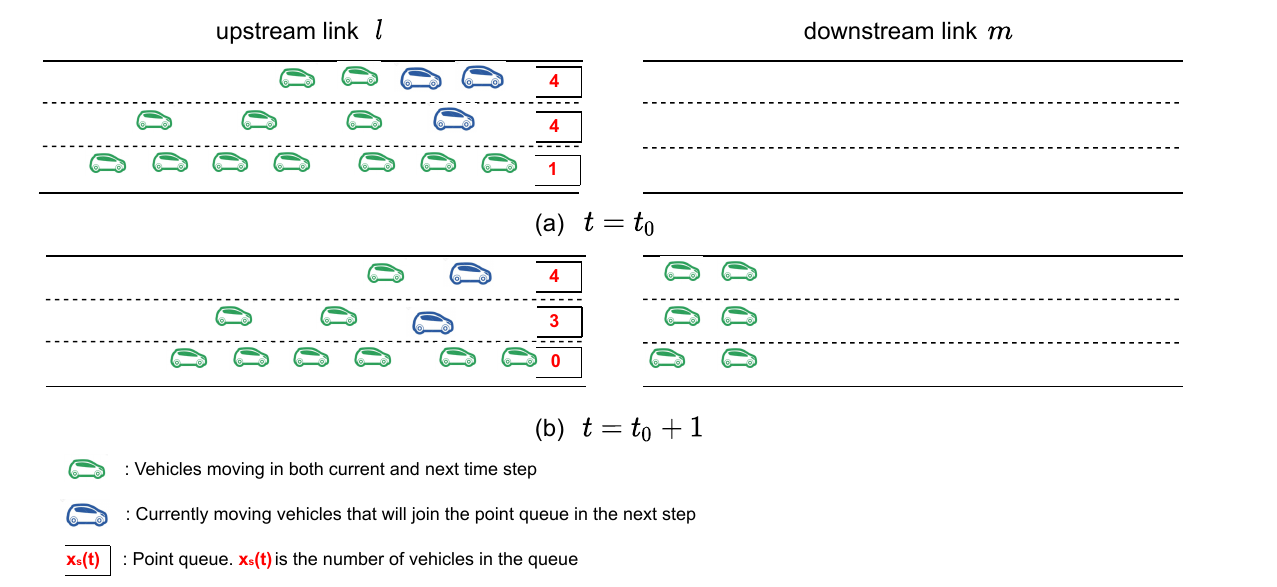}
    \caption{Evolution of traffic.}
    \label{fig:traffic_evo}
\end{figure}

\subsection{Delay based max pressure model}
This section demonstrates the proposed delay-based control policy using the store-and-forward model. First, we claim that total delay in a time step is equivalent to the number of stopped vehicles.

\begin{proposition}
Total travel delay generated on a link in a time step, $B(l,m)(t)$, is equivalent to the number of stopped vehicles, $x_s(l,m)(t)$. 
\end{proposition}
\begin{proof}
\begin{equation}
    \begin{split}
        B(l,m)(t)&=\sum_{i\in \mathbb{V}(l,m)(t)}b_i(t)=\sum_{i\in \mathbb{V}_s(l,m)(t)}b_i+\sum_{j\in \mathbb{V}_m(l,m)(t)}b_j\\
        &=\sum_{i\in \mathbb{V}_s(l,m)(t)}(\Delta t-e_i(t)/v_f)+\sum_{j\in \mathbb{V}_m(l,m)(t)}(\Delta t-(v_f\Delta t)/v_f)
    \end{split}
\end{equation}
Since $e_i(t)=0, \quad \forall i \in \mathbb{V}_s(l,m)(t)$,\\
$B(l,m)(t)=x_s(l,m)(t)\Delta t$
\end{proof}

Therefore, total delay in $T$ time steps is equivalent to the sum of the number of stopped vehicles multiplied by the time step of traffic operation. Like the Original-MP, the proposed D-MP allocates green time to maximize the overall pressure at each intersection. The pressure of each phase is defined as the sum of saturation flow multiplied by the weight over all served movements, and the weight is the difference in traffic delay between the upstream link and the average value of its downstream links incurred in the previous time step. In detail, for movement $(l,m)$, the weight is defined as
\begin{equation}\label{eq:weight}
    w(l,m)(t)=\sum_{t'=1}^{T}x_s(l,m)(t-T+t')-\sum_{n\in \mathbb{O}_m}\left[\sum_{t'=1}^{T}x_s(m,n)(t-T+t')\right]H(m,n)(t)
\end{equation}
Next, the pressure of a phase $S_j$ at intersection $i$ is defined as
\begin{equation}\label{eq:pressure}
    p(S_{ij})(t)=\sum_{(l,m)\in\mathbb{S}_{ij}}C(l,m)(t)w(l,m)(t)
\end{equation}

For any intersection $i$, at time $t$, the delay based MP selects the phase with the maximal pressure, i.e.,
\begin{equation}\label{eq:opt}
    S^*=\argmax_{j}p(S_{ij})
\end{equation}

Equations \eqref{eq:weight}-\eqref{eq:opt} are repeated every $T$ steps. Note that for each intersection, only the delay from adjacent links is needed, so this model is decentralized.

\subsection{Stability of the delay-based MP}
The key property of the original MP is its maximum stability, also known as maximization of throughput. In general, this property means if there exists a signal control policy that can accommodate a demand, then the stable control can accommodate the demand. This section will prove that the proposed D-MP is a stable control method.

\begin{definition}
A demand scenario $\{d(i),\quad i\in\mathbb{L}_e\}$ is feasible if there exists a control sequence $\boldsymbol{S}(t)$ that
\begin{equation}
    \bar{S}(l,m)c(l,m)>f_l H(l,m)\quad \forall (l,m)
\end{equation}
\end{definition}
\hl{where $f_l$ is the average volume for each link uniquely determined by the entry demand} $\boldsymbol{d}$ and turning ratios at each intersection, see \citep{varaiya2013max}, and $\bar{S}$ denotes the time ratio that each movement is served by the signal in the long-run ,
\begin{equation*}
    \bar{S}(l,m)=\liminf\limits_{T\rightarrow\infty} \frac{1}{T}\sum_{t=1}^{T}S(t)(l,m)
\end{equation*}

Therefore, a demand scenario is feasible if there exists a signal control sequence from which the average service rate for all movements in the long-run is higher than the average arrival rate.\\
\begin{definition}
    A control sequence $\boldsymbol{S}(t)$ is stable in the mean if the average number of vehicles in the network, $\frac{1}{T}\sum_{t=1}^{T}\sum_{(l,m)}\mathbb{E}[x(l,m)(t)]$, is finite for all $T$.
\end{definition}

It has been shown that stable control sequences exist if and only if the demand is feasible. The proof can be found in \citep{varaiya2013max}. Next, we prove D-MP is stable.

\begin{theorem}\label{keytheo}
The delay based MP \eqref{eq:opt} is stable if $\boldsymbol{d}\in \mathbb{D}_f$.
\end{theorem}

\begin{proof}
The stability of a control policy is established if the following inequality holds under the control policy
\begin{equation}\label{eq:keyequation}
    \mathbb{E}\{|\boldsymbol{X}(t+T)|^2-|\boldsymbol{X}(t)|^2|\bm{X}(t)\}\le k-\tau |\boldsymbol{X}(t)|, t=0,T,2T,...
\end{equation}
where $|\boldsymbol{X}|^2$ is the sum of squares of each element in $\boldsymbol{X}$, i.e., $|\boldsymbol{X}|^2=\sum_{(l,m)}x^2(l,m)$; $k<\infty$ and $\tau>0$ are two constants.\\
The proof can be found in \citep{varaiya2013max}, but it is repeated here to make the proof easier to follow.

Unconditioning the expectation and summing over time $t=0,T,2T,...,nT$ gives
\begin{equation}
    \mathbb{E}|\boldsymbol{X}(nT)|^2-\mathbb{E}|\boldsymbol{X}(0)|^2\le nk-\tau\sum_{m=1}^n\mathbb{E}|\boldsymbol{X}(mT)|
\end{equation}

Then,
\begin{equation}\label{eq:conclusion}
    \tau\frac{1}{n}\sum_{m=1}^n\mathbb{E}|\boldsymbol{X}(mT)|\le k+\frac{1}{n}\mathbb{E}|\boldsymbol{X}(1)|^2
\end{equation}

Equation \eqref{eq:conclusion} indicates the average number of vehicle in the network through any time is upper bounded, so Theorem \ref{keytheo} is proved. Therefore, Equation \eqref{eq:keyequation} is the key for the proof, which is proved by the following steps.\\

First, let $\boldsymbol{\delta}=\boldsymbol{X}(t+T)-\boldsymbol{X}(t)$ denote the change in the number of vehicles between signal changes. Then,
\begin{equation}
    |\boldsymbol{X}(t+T)|^2-|\boldsymbol{X}(t)|^2=2\boldsymbol{X}(t)^T\boldsymbol{\delta}+|\boldsymbol{\delta}|^2=2\alpha+\beta
\end{equation}

Next, we are going to prove both $\alpha$ and $\beta$ are upper bounded. \\
\begin{lemma}\label{lemma:alpha}
$\alpha$ is upper bounded by $-\tau|\boldsymbol{X}(t)|+constant$, where $\tau>0$.
\end{lemma}

\begin{proof}
By substituting Equations \eqref{eq:stop_evo}-\eqref{eq:move_in_evo} into the expression of $\alpha$, it is proved in \citep{varaiya2013max} that
\begin{equation}
    \mathbb{E}\{\alpha|x(t)\}=\sum_{l\in\mathbb{L}, m}[d_lH(l,m)(t)T-\mathbb{E}\{\min(C(l,m)(t)S^*(l,m)(t)T, x(l,m)(t)|x(t)\}]w'(l,m)(t)=\alpha_1T+\alpha_2
\end{equation}
where 
\begin{equation}\label{eq:queueweight}
    w'(l,m)(x(t))=x(l,m)(t)-\sum_{n\in \mathbb{O}_m}x(m,n)(t)H(m,n)(t)
\end{equation}
is the weight from Original-MP, and 
\begin{equation}\label{eq:alpha1}
   \alpha_1=\sum_{l\in\mathbb{L}, m}[d_lH(l,m)-c(l,m)(t)S^*(l,m)(t)]w'(l,m)(t)
\end{equation}
\begin{equation}\label{eq:alpha2}
    \alpha_2=\sum_{l\in\mathbb{L}, m}[c(l,m)(t)T-\mathbb{E}\{\min(C(l,m)(t)T, x(l,m)(t)|x(t)\}]S^*(l,m)(t)w'(l,m)(t)
\end{equation}

\hl{Note although Equations} \eqref{eq:alpha1} and \eqref{eq:alpha2} \hl{share the same form with} \citep{varaiya2013max}, \hl{the relationship between $S^*(l,m)(t)$ and $w'(l,m)$ differs. In} \citep{varaiya2013max}, \hl{$S^*(l,m)(t)$ is the phase according to the number of vehicle-based weight $w'(l,m)$. On the contrary, in our model, $S^*(l,m)(t)$ is the phase according to delay-based weight} \eqref{eq:weight} \hl{rather than $w'(l,m)$, which distinguishes the following proof different from} \citep{varaiya2013max}. Next, we prove $\alpha_2$ is upper bounded.

It is straightforward that the term in the square bracket in Equation \eqref{eq:alpha2} is always positive, $S^*(l,m)(t)$ is a 0-1 function , and according to \eqref{eq:queueweight}, $w'(l,m)(t)\le x(l,m)$
\begin{equation*}
    \alpha_2\le\sum_{l\in\mathbb{L}, m}[c(l,m)(t)T-\mathbb{E}\{\min(C(l,m)(t)T, x(l,m)(t)|x(t)\}]S^*(l,m)(t)x(l,m)(t)
\end{equation*}
In addition,
\begin{equation*}
    c(l,m)(t)T-\mathbb{E}\{\min(C(l,m)(t)T, x(l,m)(t)|x(t)\}:
    \begin{cases}
      =0 & \text{if $x(l,m)(t)\ge C(l,m)(t+1)$}\\
      \le c(l,m)(t)T & \text{otherwise}\\
    \end{cases}       
\end{equation*}

Therefore, $\alpha_2\le c(l,m)\bar{C}(l,m)T$ where $c(l,m)$ and $\bar{C}(l,m)$ are the mean and the upper bound of the saturation flow, respectively. Consequently, $\alpha_2$ is upper bounded by a constant.\\
Next, we will prove $\alpha_1$ is upper bounded.\\
Combining Equations \eqref{eq:agg}, \eqref{eq:queueweight} and \eqref{eq:alpha1} provides
\begin{equation}
\begin{split}
    \alpha_1=&\sum_{l\in\mathbb{L}, m}[d_lH(l,m)-c(l,m)(t)S^*(l,m)(t)]\left[x_m(l,m)(t)-\sum_{n\in \mathbb{O}_m}x_m(m,n)(t)H(m,n)(t)\right]\\
    &+\sum_{l\in\mathbb{L}, m}[d_lH(l,m)-c(l,m)(t)S^*(l,m)(t)]\left[x_s(l,m)(t)-\sum_{n\in \mathbb{O}_m}x_s(m,n)(t)H(m,n)(t)\right]\\
    =&\alpha_{11}+\alpha_{12}
\end{split}
\end{equation}

As mentioned before, the number of moving vehicles at any time \hl{$x_m(l,m)(t)$ is upper bounded by a constant,} and $d_lH(l,m)-c(l,m)(t)S^*(l,m)(t)$ is also upper bounded by a constant, so $\alpha_{11}$ is upper bounded by a constant. Next, we will prove $\alpha_{12}$ is also upper bounded. \hl{First, we introduce $w(l,m)(t)$, defined in Equation} \eqref{eq:weight}, \hl{into $\alpha_{12}$ by}
\begin{equation}\label{eq:alpha12}
    \begin{split}
        \alpha_{12} = &\sum_{l\in\mathbb{L}, m}[d_lH(l,m)-c(l,m)(t)S^*(l,m)(t)]\left[\sum_{t'=0}^{T-1}\left[x_s(l,m)(t-t')-\sum_{n\in \mathbb{O}_m}x_s(m,n)(t-t')H(m,n)(t-t')\right]\right]\\
        & -\sum_{l\in\mathbb{L}, m}[d_lH(l,m)-c(l,m)(t)S^*(l,m)(t)]\left[\sum_{t'=1}^{T-1}\left[x_s(l,m)(t-t')-\sum_{n\in \mathbb{O}_m}x_s(m,n)(t-t')H(m,n)(t-t')\right]\right]\\
        = & \sum_{l\in\mathbb{L}, m}[d_lH(l,m)-c(l,m)(t)S^*(l,m)(t)]w(t)\\
        & -\sum_{l\in\mathbb{L}, m}[d_lH(l,m)-c(l,m)(t)S^*(l,m)(t)]\left[\sum_{t'=1}^{T-1}\left[x_s(l,m)(t-t')-\sum_{n\in \mathbb{O}_m}x_s(m,n)(t-t')H(m,n)(t-t')\right]\right]\\
    \end{split}
\end{equation}

Next, we prove the first term in \eqref{eq:alpha12} is upper bounded. Let
\begin{equation}\label{eq:alpha_prime}
    \alpha'=\sum_{l\in\mathbb{L}, m}[d_lH(l,m)(t)-c(l,m)(t)S^*(l,m)(t)]w(t)
\end{equation} \\
Since $d\in\mathbb{D}_f$, there exist a signal control matrix $S^+\in co(S)$ and $\epsilon>0$ such that $c(l,m)S^+(l,m)>d_lH(l,m)+\epsilon\quad \forall (l,m)$, where $co(S)$ is the convex hull of all possible signal timing. Therefore, we can find $S_1\in co(S)$ such that
\begin{equation}\label{eq:constructofsignal}
    c(l,m)S_1(l,m)(t)=
    \begin{cases}
    d_lH(l,m)(t)+\epsilon, & \text{if $w(l,m)(t)\ge 0$}\\
    0, & \text{otherwise}
    \end{cases}
\end{equation}
Since $S^*$ maximizes the term of $c(l,m)(t)S^*(l,m)(t)w(t)$,
\begin{equation}
    \begin{split}
        \alpha'&\le \sum_{l\in\mathbb{L}, m}[d_lH(l,m)-c(l,m)(t)S_1(l,m)(t)]w(t)\\
        &=\sum_{l\in\mathbb{L}, m}[d_lH(l,m)-d_lH(l,m)(t)-\epsilon]w^+(l,m)(t)+\sum_{l\in\mathbb{L}, m}[d_lH(l,m)]w^-(l,m)(t)\\
        &\le -\epsilon\sum_{l\in\mathbb{L}, m}|w(t)|
    \end{split}
\end{equation}
where $w^+=\max\{0,w\}$ and $w^-=\min\{0,w\}$. According to Equation \eqref{eq:weight}, $w(l,m)(t)$ is a linear function of $\{x_s(l,m)(t-T+1), x_s(l,m)(t-T+2),...,x_s(l,m)(t)\}$, so there exists a constant $\eta_1>0$ such that $\sum_{l\in\mathbb{L}, m}|w(t)|\ge \eta_1\sum_{l\in\mathbb{L}, m}|\bm{X}_s(t)|$ where $\bm{X}_s(t)=[x_s(l,m)(t-T+1), x_s(l,m)(t-T+2),...,x_s(l,m)(t)]$. Similarly, according to \eqref{eq:stop_evo} and the assumption that \hl{$\Psi_1(\boldsymbol{V}(l,m)(t))$ is upper bounded by a constant}, $x_s(t-1)$ can be expressed as the sum of $x_s(t)$ and another term that is upper bounded by a constant. Therefore, there exists a constant $\eta_2>0$ such that $\sum_{l\in\mathbb{L}, m}|\bm{X}_s(t)|\ge\eta_2\sum_{l\in\mathbb{L}, m}x_s(l,m)(t)$. Therefore,
\begin{equation}\label{eq:alpha_prime_ineq}
    \alpha'\le -\epsilon\eta_1\eta_2\sum_{l\in\mathbb{L}, m}x_s(l,m)(t)\le -\epsilon\eta_1\eta_2\sum_{l\in\mathbb{L}, m}x(l,m)(t)
\end{equation}

\hl{Next, based on Equation} \eqref{eq:alpha_prime_ineq}, \hl{we prove that the second term in Equation} \eqref{eq:alpha12} \hl{is also upper bounded}. According to Equation \eqref{eq:stop_evo},
\begin{equation}
\begin{split}
    x_s(l,m)(t-1)&=x_s(l,m)(t)-\Psi_1(\boldsymbol{V}(l,m)(t))+\min\{C(l,m)(t+1)S(l,m)(t+1), x_s(l,m)(t)+\Psi_1(\boldsymbol{V}(l,m)(t))\}\\
    &=x_s(l,m)(t)+z(l,m)(t-1)
\end{split}
\end{equation}
where $z(l,m)(t-1)$ is finite \hl{since we assume $\Psi_1(\boldsymbol{V}(l,m)(t))$ is upper bounded by a constant}. Similarly,
\begin{equation}\label{eq:xs}
    x_s(l,m)(t-t')=x_s(l,m)(t)+z(l,m)(t-t'), t'=1,2,...,T-1
\end{equation}
Then, replacing $x_s(l,m)(t-t')$ and $x_s(m,n)(t-t')$ in Equation \eqref{eq:alpha12} with \eqref{eq:xs} provides
\begin{equation}
    \begin{split}
        \alpha_{12} = &\alpha' -\sum_{l\in\mathbb{L}, m}[d_lH(l,m)-c(l,m)(t)S^*(l,m)(t)]\left[\sum_{t'=1}^{T-1}\left[x_s(l,m)(t-t')\right]-\sum_{t'=1}^{T-1}\left[\sum_{n\in \mathbb{O}_m}x_s(m,n)(t-t')H(m,n)(t-t')\right]\right]\\
        \le & -\epsilon\eta_1\eta_2\sum_{l\in\mathbb{L}, m}x(l,m)(t)-\sum_{l\in\mathbb{L}, m}[d_lH(l,m)(t)-c(l,m)(t)S^*(l,m)(t)]\left[\sum_{t'=1}^{T-1}\left[x_s(l,m)(t)+z(l,m)(t-t')\right]\right.\\
        &\left. -\sum_{t'=1}^{T-1}\left[\sum_{n\in \mathbb{O}_m}x_s(m,n)(t)H(m,n)(t)+z(m,n)(t-t')\right]\right]\\
        =& -\epsilon\eta_1\eta_2\sum_{l\in\mathbb{L}, m}x(l,m)(t)-\sum_{l\in\mathbb{L}, m}[d_lH(l,m)(t)-c(l,m)(t)S^*(l,m)(t)]\left[\sum_{t'=1}^{T-1}\left[x_s(l,m)(t)-\sum_{n\in \mathbb{O}_m}x_s(m,n)(t)H(m,n)(t)\right]\right.\\
        &\left. +\sum_{t'=1}^{T-1}\left[z(l,m)(t-t')-\sum_{n\in \mathbb{O}_m}z(m,n)(t-t')\right]\right]\\
        =&-\epsilon\eta_1\eta_2\sum_{l\in\mathbb{L}, m}x(l,m)(t)-(T-1)\sum_{l\in\mathbb{L}, m}[d_lH(l,m)(t)-c(l,m)(t)S^*(l,m)(t)]\left[x_s(l,m)(t)-\sum_{n\in \mathbb{O}_m}x_s(m,n)(t)H(m,n)(t)\right]\\
        & -\sum_{l\in\mathbb{L}, m}[d_lH(l,m)(t)-c(l,m)(t)S^*(l,m)(t)]\left[\sum_{t'=1}^{T-1}\left[z(l,m)(t-t')-\sum_{n\in \mathbb{O}_m}z(m,n)(t-t')\right]\right]\\
        =& -\epsilon\eta_1\eta_2\sum_{l\in\mathbb{L}, m}x(l,m)(t)-(T-1)\alpha_{12}\\
        & -\sum_{l\in\mathbb{L}, m}[d_lH(l,m)(t)-c(l,m)(t)S^*(l,m)(t)]\left[\sum_{t'=1}^{T-1}\left[z(l,m)(t-t')-\sum_{n\in \mathbb{O}_m}z(m,n)(t-t')\right]\right]\\
    \end{split}
\end{equation}
By moving the second term to the left-hand side, we obtain
\begin{equation}
    \begin{split}
        T\alpha_{12}\le &-\epsilon\eta_1\eta_2\sum_{l\in\mathbb{L}, m}x(l,m)(t)\\
        &-\sum_{l\in\mathbb{L}, m}[d_lH(l,m)(t)-c(l,m)(t)S^*(l,m)(t)]\left[\sum_{t'=1}^{T-1}\left[z(l,m)(t-t')-\sum_{n\in \mathbb{O}_m}z(m,n)(t-t')\right]\right]\\
    \end{split}
\end{equation}
Since the second term on the right-hand side is finite, we obtain,
\begin{equation}
    \alpha_{12}\le -\epsilon\eta_1\eta_2\sum_{l\in\mathbb{L}, m}x(l,m)(t)+constant
\end{equation}
Let $\tau=\epsilon\eta_1\eta_2$, Lemma \ref{lemma:alpha} is proved.
\end{proof}
\begin{lemma}
$\beta$ is upper bounded by a constant.
\end{lemma}
\begin{proof}
Based on the store-and-forward models \eqref{eq:stop_evo}-\eqref{eq:move_in_evo}, it can be easily seen that the difference in the number of two consecutive time steps is upper bounded by a constant. Therefore, $\beta=|\delta|^2$ is upper bounded a constant. The detailed proof for this Lemma can be found in \citep{varaiya2013max}.
\end{proof}
Above all, Equation \eqref{eq:keyequation} is proved. Consequently, Theorem \ref{keytheo} is proved.
\end{proof}

\section{Case study}\label{sc:cs}
In addition to Original-MP and TT-MP, one may question how the Original-MP performs if we replace the number of vehicles (as indicated by `queue' in the Original-MP) with the number of stopped vehicles (the real queue). To make the comparison comprehensive, we include such model in the following, and it is referred to as halting MP (H-MP). Since Original-MP and H-MP use instantaneous measures while the TT-MP and D-MP use average metrics over a time step, they are also collectively referred to as instantaneous-based model and average-based models, respectively. Therefore, this section investigates the control performance of four MP models: Original-MP, TT-MP, H-MP, and D-MP under various demand scenarios in the  microsimulator SUMO \citep{SUMO2018}. Before moving to the network settings, we first clarify how the used metrics are obtained from  SUMO. For the Original-MP, the number of vehicles on a link at a time instant can be retrieved directly from SUMO; for H-MP, the number of halting vehicles, which are defined as the vehicles with speed lower than 0.1 m/s, at a time instant can be retrieved directly from SUMO; for TT-MP, we implement a time step of 1s in SUMO, so the total travel time in the $i$th interval [(i-1)T,iT] is equal to the sum of number of vehicles on the link over all time steps in this interval; for D-MP, we previously assumed (for simplicity) that all moving vehicles are traveling at free flow speed. To be more realistic, we relax this assumption in the simulation, and the total delay in the $i$th time step is defined as:
\begin{equation}
    \bm{B}(l,m)(i) = \bm{tt}(l,m)(i)-\frac{\bm{e}(l,m)(i)}{v_f}
\end{equation}
where the travel distance $\bm{e}(l,m)(i)$ is equal to the sum of average speed multiplied by the number of vehicles over all time steps in the $i$th interval.

We simulate a  4$\times$4 grid network. Each link is assumed to have two lanes, one for the left turning movement and the other for a shared through and right turn movement. For simplicity and without lack of genearlity, the turning ratios are assumed to be the same at all intersections:  0.2 (left), 0.3 (right) and 0.5 (through). The demand is symmetric between the north-south and east-west entry links. The north-south entry links' demand is assumed to be twice that of the east-west entry links. Figure \ref{fig:network} shows the network configuration.  Three scenarios with different degree of saturation were tested. Note that the saturation flow is not an input that can be defined by users in SUMO. Instead, it is determined by the vehicle parameters, lane parameters and car-following models. The max speed for all links is set to 20 m/s; the vehicle length is 5 m; the maximum acceleration is 20 m/s$^2$; the maximum deceleration is 4.5 m/s$^2$, and the default Krauss car-following model \citep{krauss1998microscopic} is used. To calibrate the saturation flow, we simulated a very high demand for a single movement at an isolated intersection and activated its phase for the length of the  simulation. The  hourly throughput under this setting revealed the associated  saturation flow. Using this method, the saturation flows for both left turns and through/right turns are obtained to be 1,800 veh/h/lane. 

\begin{figure}[!htbp]
    \centering
    \includegraphics[width=5.5in]{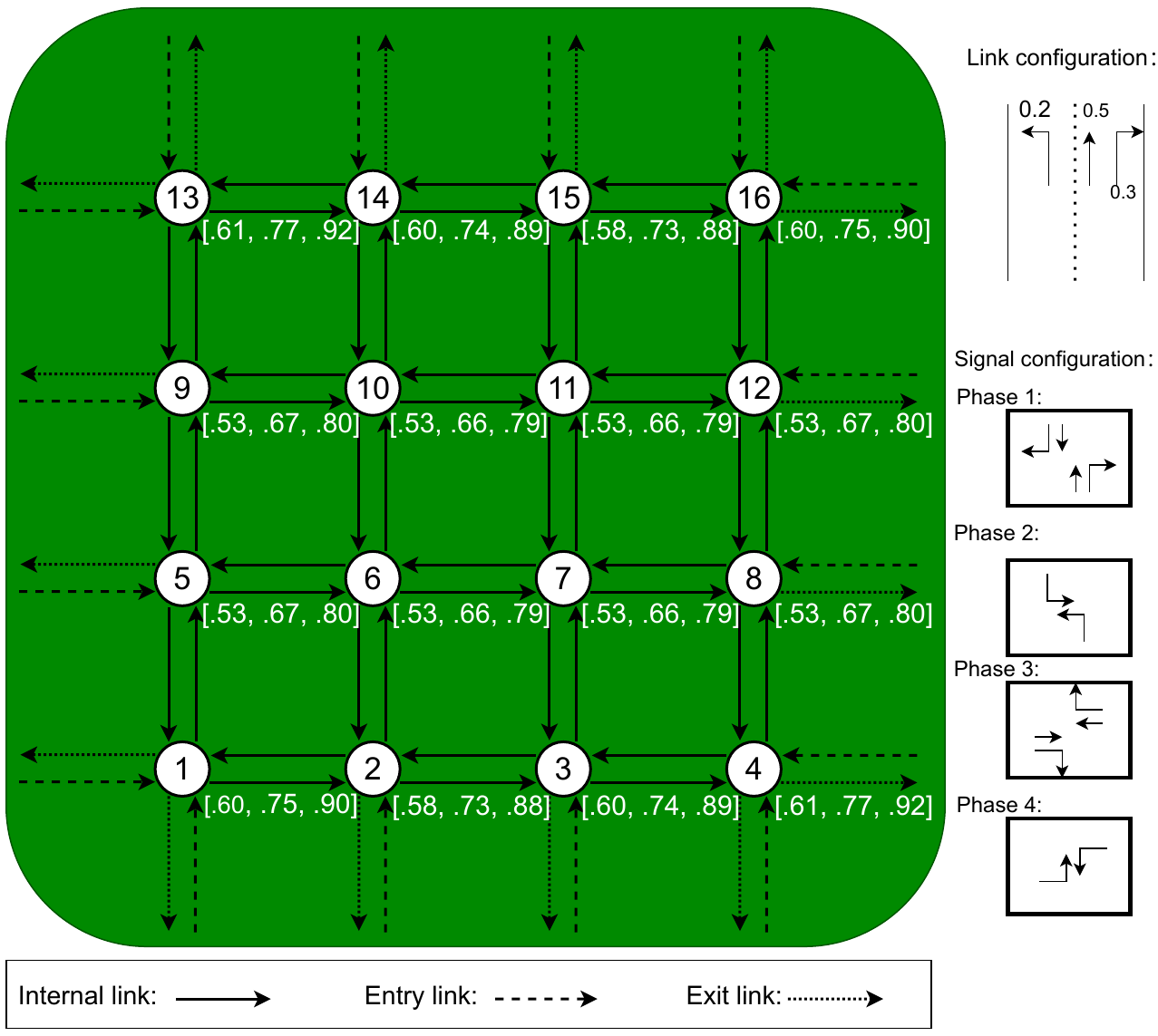}
    \caption{Network configuration.}
    \label{fig:network}
\end{figure}

\subsection{Time step size}\label{sec:timestep}
Although the time step $T$ is equal to the free flow speed in the store-and-forward model, it cannot be proved to be optimal for any metric. A long time step may lose resolution and lead to poor performance. For example, it is known that shorter green time (and green time updates) are desirable when traffic volumes are low. Therefore, the free flow travel time might be too long under low traffic conditions. On the other hand, the saturation flow in Equations \eqref{eq:discharge}-\eqref{eq:move_in_evo} is adjusted by the lost time, assumed to be 3s in this paper. This means when we calculate the pressure for the phase that is different from the previous step, the saturation flow is equal to $C_o\frac{T-3}{T}$, where $C_0$ is the real saturation flow. Therefore, if $T$ is too short, the pressure for these phases can be considerably underestimated and the algorithm barely switches the signal phase, which can lead to poor performance as well. Therefore, before going further, simulation runs were used to obtain the optimal time step for the proposed model under different scenarios. Three time-invariant demands were considered. The demands for north-south entry links were equal to 600 veh/h, 750 veh/h and 900 veh/h in these scenarios. Given the demand at the entry links and the turning ratios, the average flow for internal and exit links can be computed as,
\begin{equation}
    \boldsymbol{f}_{in} = (\boldsymbol{I}-\boldsymbol{H}_{in\rightarrow in})\boldsymbol{H}_{en\rightarrow in}\boldsymbol{d}
\end{equation}
and
\begin{equation}
    \boldsymbol{f}_{ex} = \boldsymbol{H}_{en\rightarrow ex}\boldsymbol{d}+\boldsymbol{H}_{in\rightarrow ex}(\boldsymbol{I}-\boldsymbol{H}_{in\rightarrow in})\boldsymbol{H}_{en\rightarrow in}\boldsymbol{d}
\end{equation}
where $\boldsymbol{f}_{in}$ and $\boldsymbol{f}_{ex}$ represent the average flow vector for interval and exit links, respectively; $\boldsymbol{H}_{in\rightarrow in}$, $\boldsymbol{H}_{en\rightarrow in}$, $\boldsymbol{H}_{en\rightarrow ex}$ and $\boldsymbol{H}_{in\rightarrow ex}$ are the turning matrices from internal links to internal links, from entry links to internal links, from entry links to exit links and from internal links to exit links, respectively. Details can be found in \citep{hao2018model, varaiya2013max1}. After obtaining the average flow rate for each link, the degree of saturation for each intersection was calculated by summing the critical degree of saturation over all four phases. The three numbers in a bracket next to an intersection are the calculated degree of saturation in the low, medium and high demand scenarios, respectively.

For each scenario, ten runs with different random starting seeds were performed. Figure \ref{fig:influence_ts} shows the impact on time step size on three metrics: average delay per vehicle, average queue length per link and average throughput for D-MP. The average delay per vehicle was output by SUMO directly. For queue lengths, the average queue length was found using the actual queue length over all links at each 1s simulation time period and dividing by the simulation time in seconds and the number of links. The average throughput was equal to  the total number  vehicles exiting the network divided by the simulation time. The results in Figure \ref{fig:influence_ts} reveal that for  low demand and medium demand scenarios, 5s is optimal for all metrics. For the high demand scenario, 6s is optimal, but the difference from 5s is relatively small. Therefore, we chose 5s as the time step for the proposed model in the following. The corresponding results for Original-MP, H-MP and TT-MP are shown in Appendix \ref{app:ts}. The optimal time steps for Original-MP, H-MP and TT-MP are 9s, 5s and 9s, respectively. To have a fair comparison, we use the optimal time step for each  models in the following comparison to ensure the best performance of each.

\begin{figure}[!htbp]
     \centering
     \begin{subfigure}{0.6\textwidth}
         \centering
         \includegraphics[width=\textwidth]{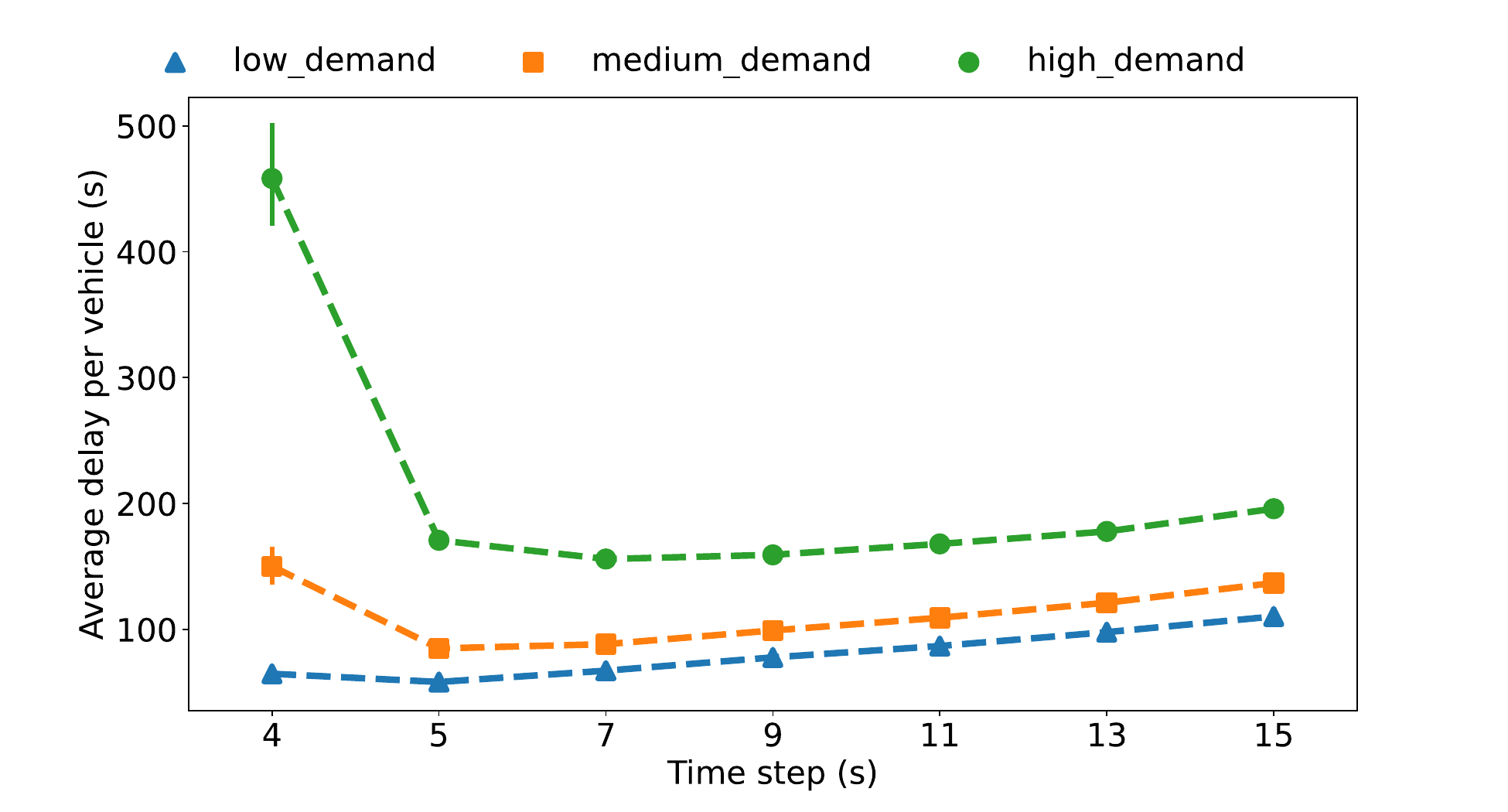}
         \caption{Average delay per vehicle.}
         \label{fig:delay_ts}
     \end{subfigure}
     \begin{subfigure}{0.6\textwidth}
         \centering
         \includegraphics[width=\textwidth]{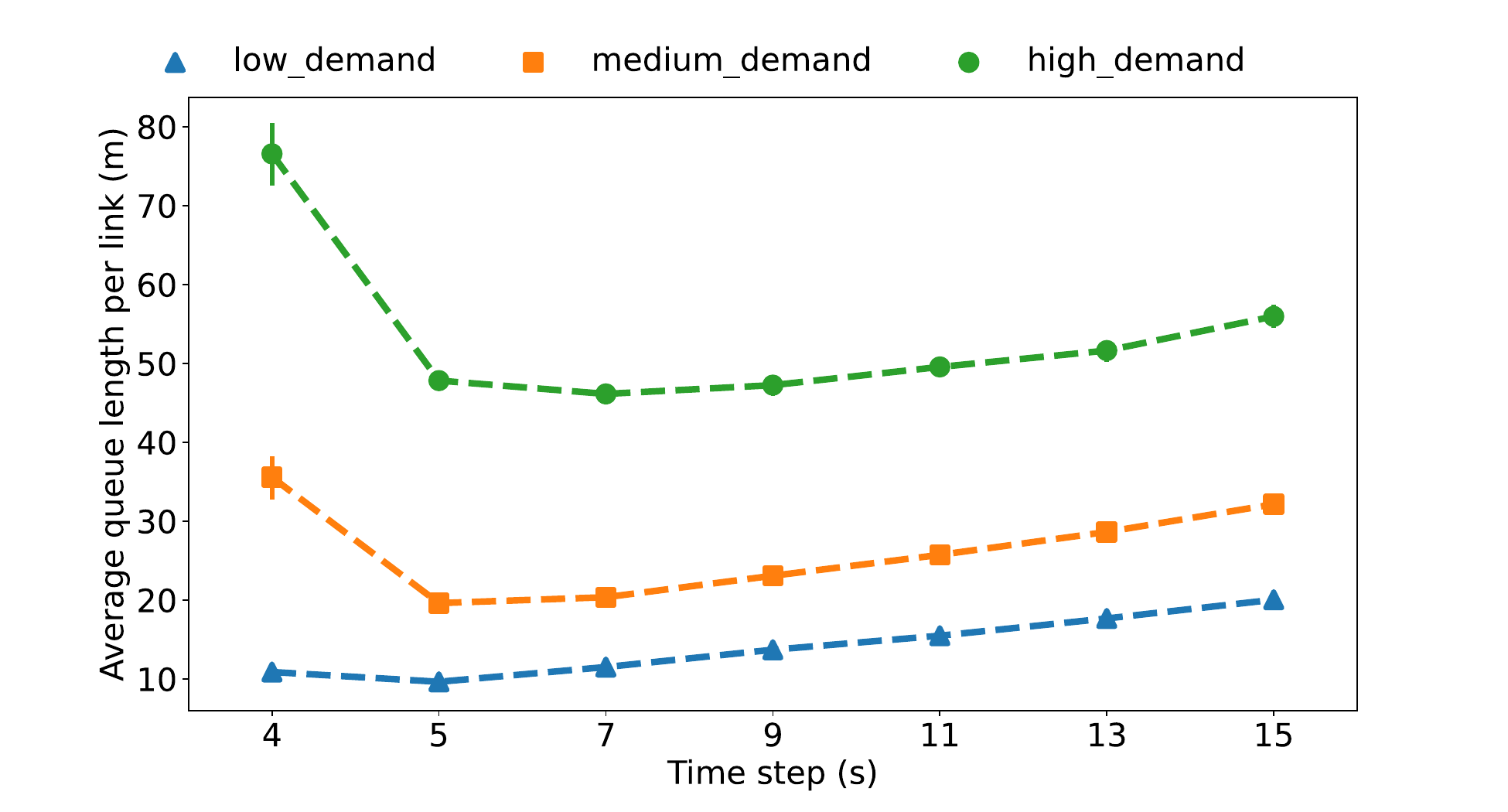}
         \caption{Average queue length}
         \label{fig:queue_ts}
     \end{subfigure}
     \hfill
     \begin{subfigure}{0.6\textwidth}
         \centering
         \includegraphics[width=\textwidth]{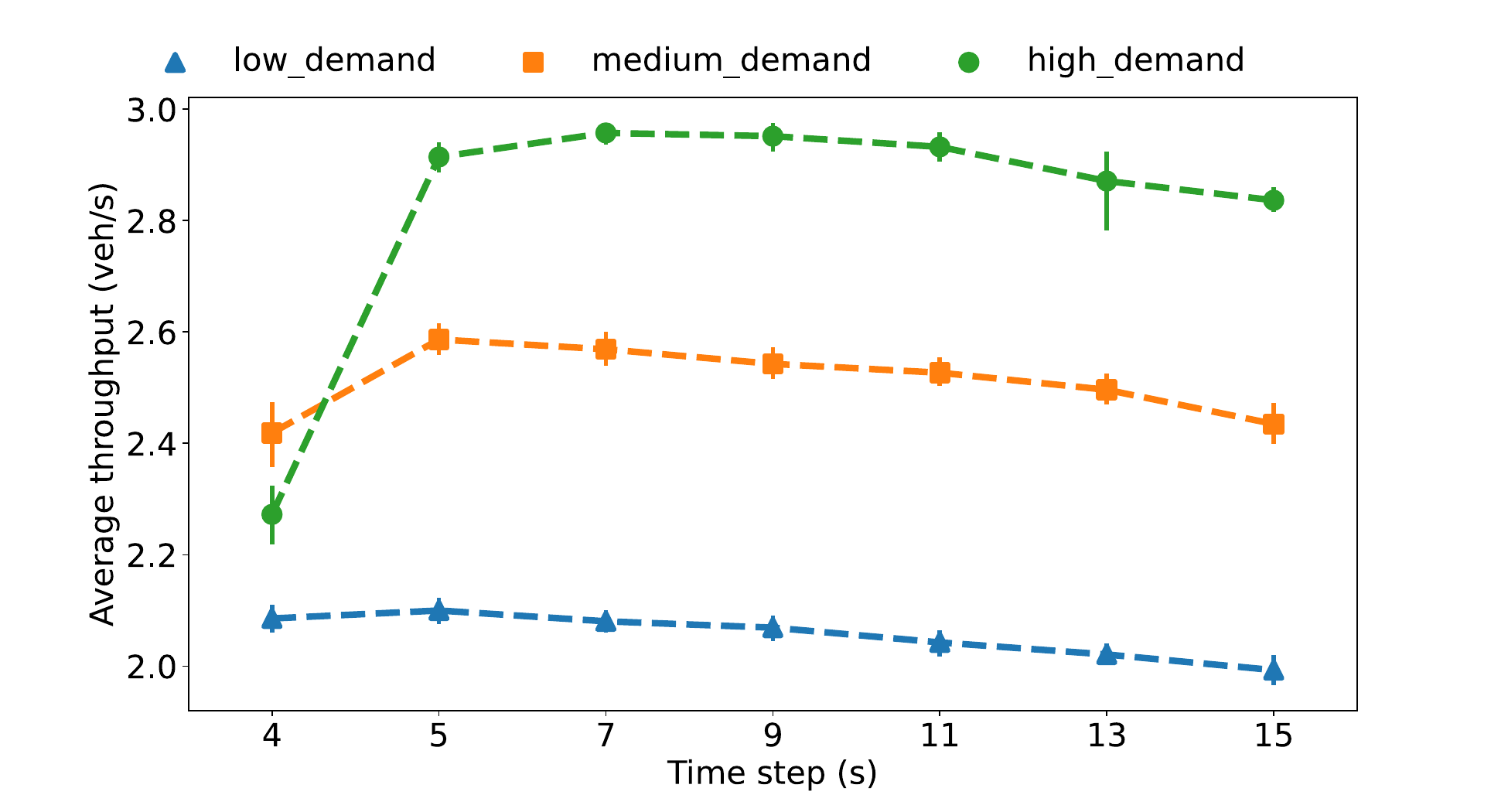}
         \caption{Average throughput}
         \label{fig:through_ts}
     \end{subfigure}
     \caption{Effect of time step on control performance.}
     \label{fig:influence_ts}
\end{figure}

\subsection{Control comparison on time varying traffic conditions}
Using the time steps obtained from the previous section, we implemented the four models on the same network under a more realistic traffic pattern. The total simulation time is 4 hours. Travel demand was assumed to be low (equal to the low demand in Section \ref{sec:timestep}) in the first half hour and then increased continuously to reach the high demand (equal to the high demand in Section \ref{sec:timestep}) during the course of one hour. Then, the demand is held constant for one hour and decreases back to the low level of the course of an hour. The demand for north-south entry links was kept to be  twice that of  east-west entry links. Figure \ref{fig:demand} shows the demand pattern for both types of entry links. 10 runs with different random seeds were performed and the shaded area indicates the variability (using the standard deviation) in the simulation runs. 

\begin{figure}[!htbp]
    \centering
    \includegraphics[width=3.5in]{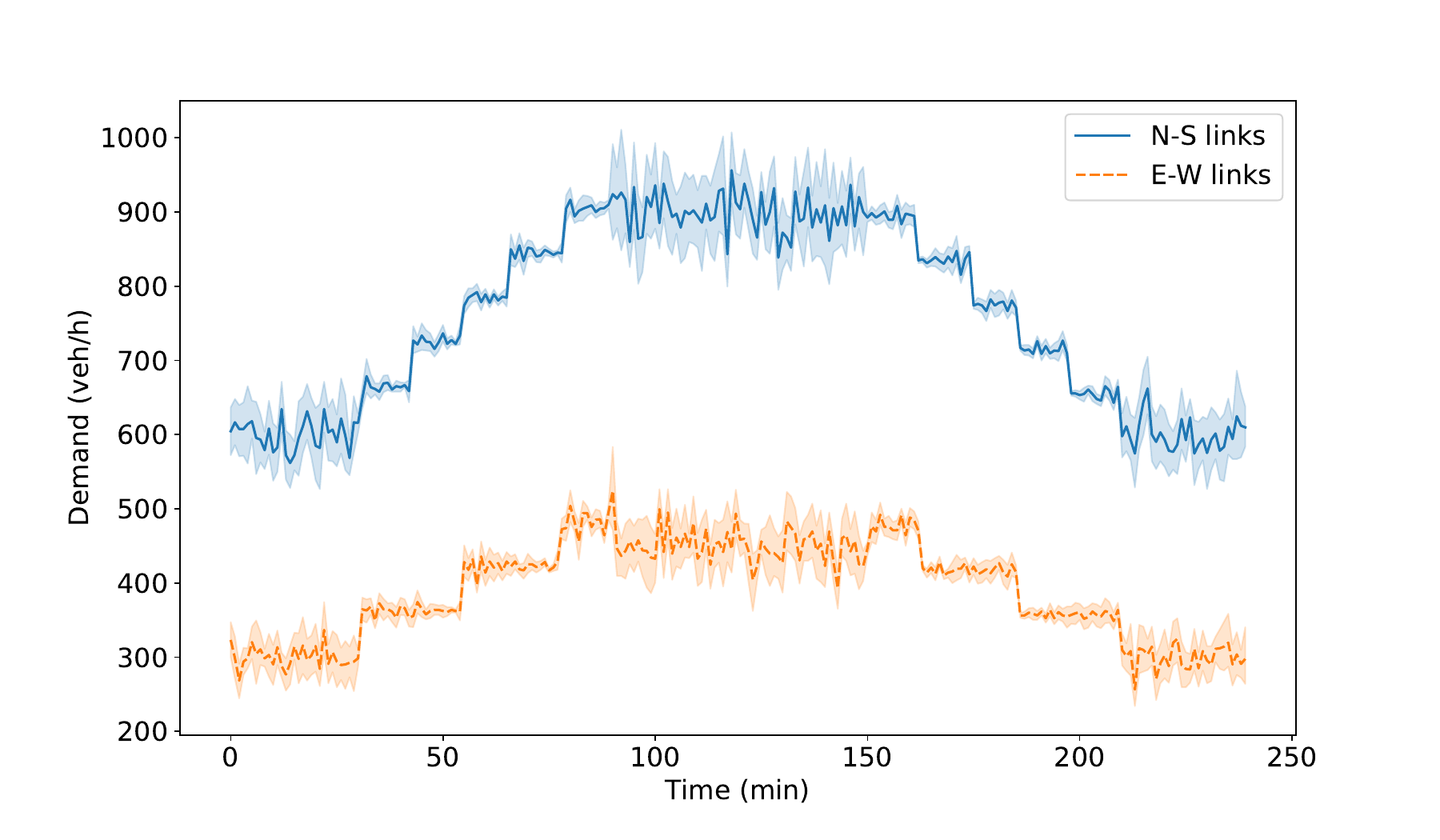}
    \caption{Demand.}
    \label{fig:demand}
\end{figure}

\subsubsection{Internal delay}
Figure \ref{fig:comparison_delay} shows the comparison of internal delay produced by the  four models, which is defined as the average delay incurred by the vehicles already in the network. First of all, D-MP has the best performance in terms of total delay, as shown by Figure \ref{fig:delay_cumulative}. Quantitatively, compared to Original-MP, TT-MP and H-MP, D-MP reduces delay by 19.12\%/330.87 hours, 11.54\%/182.54 hours and 3.95\%/57.49 hours, respectively. Figure \ref{fig:delay_min} shows that both instantaneous-based models perform slightly better than the average-based models when the traffic condition is low (before 50 minutes). As traffic demand increases, the average-based models outperform the instantaneous-based models. After demand returns the low level (after 210 minutes), the average-based models mitigate delay significantly faster than the instantaneous-based models. Overall, Figure \ref{fig:delay_min} demonstrates that the number of stopped vehicles is a better choice in terms of reducing traffic delay than the number of vehicles, and the average metric is beneficial for heavy traffic conditions. Figure \ref{fig:internal_veh} shows the number of vehicles in the network, which has a similar pattern with the delay shown in Figure \ref{fig:delay_min}. This plot proves the average density in the network from D-MP is the lowest for majority of the simulation period; this results in higher average speeds and lower delays.

\begin{figure}[!htbp]
     \centering
     \begin{subfigure}{0.45\textwidth}
         \centering
         \includegraphics[width=\textwidth]{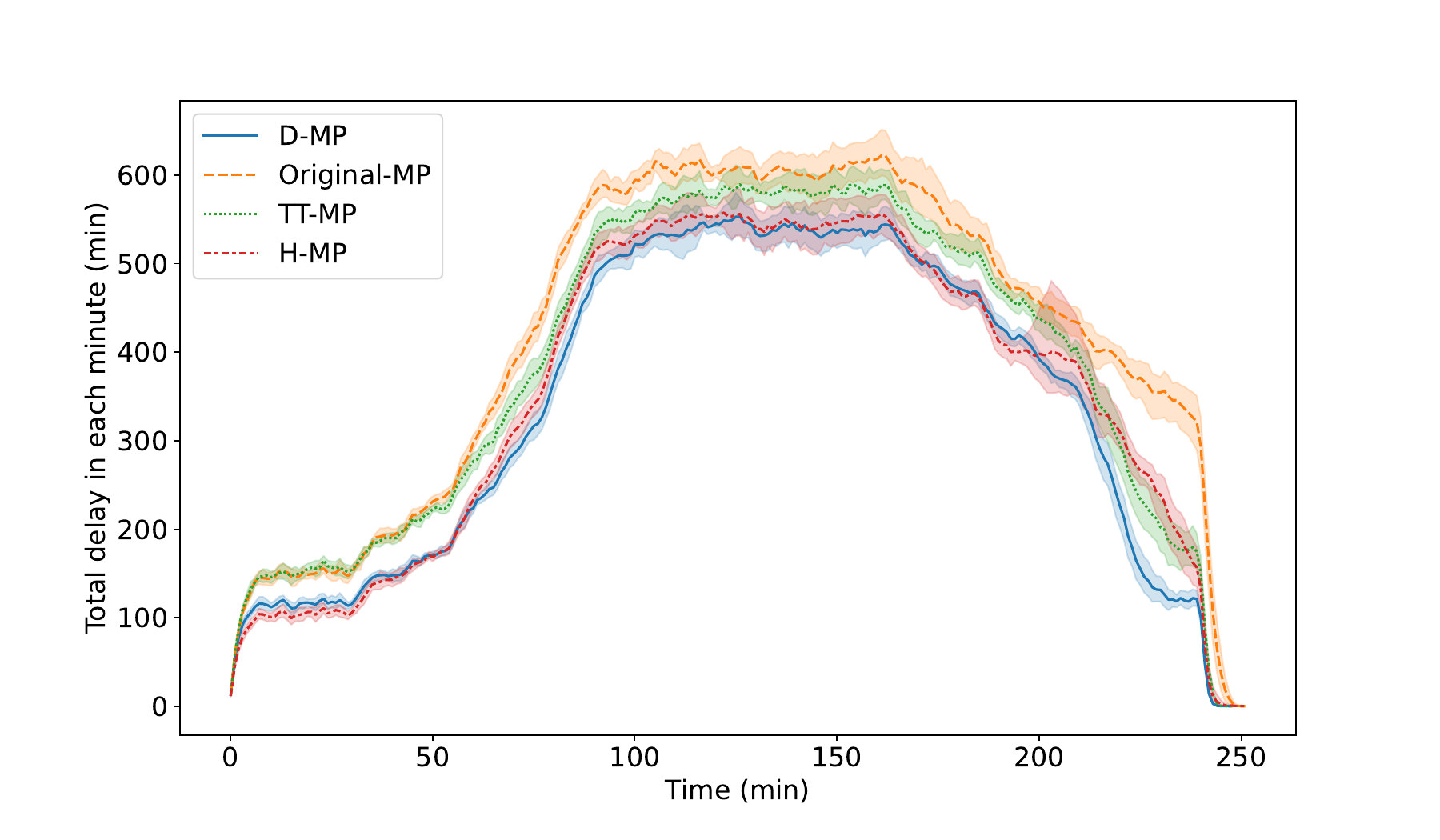}
         \caption{Delay incurred in each minute.}
         \label{fig:delay_min}
     \end{subfigure}
     \begin{subfigure}{0.45\textwidth}
         \centering
         \includegraphics[width=\textwidth]{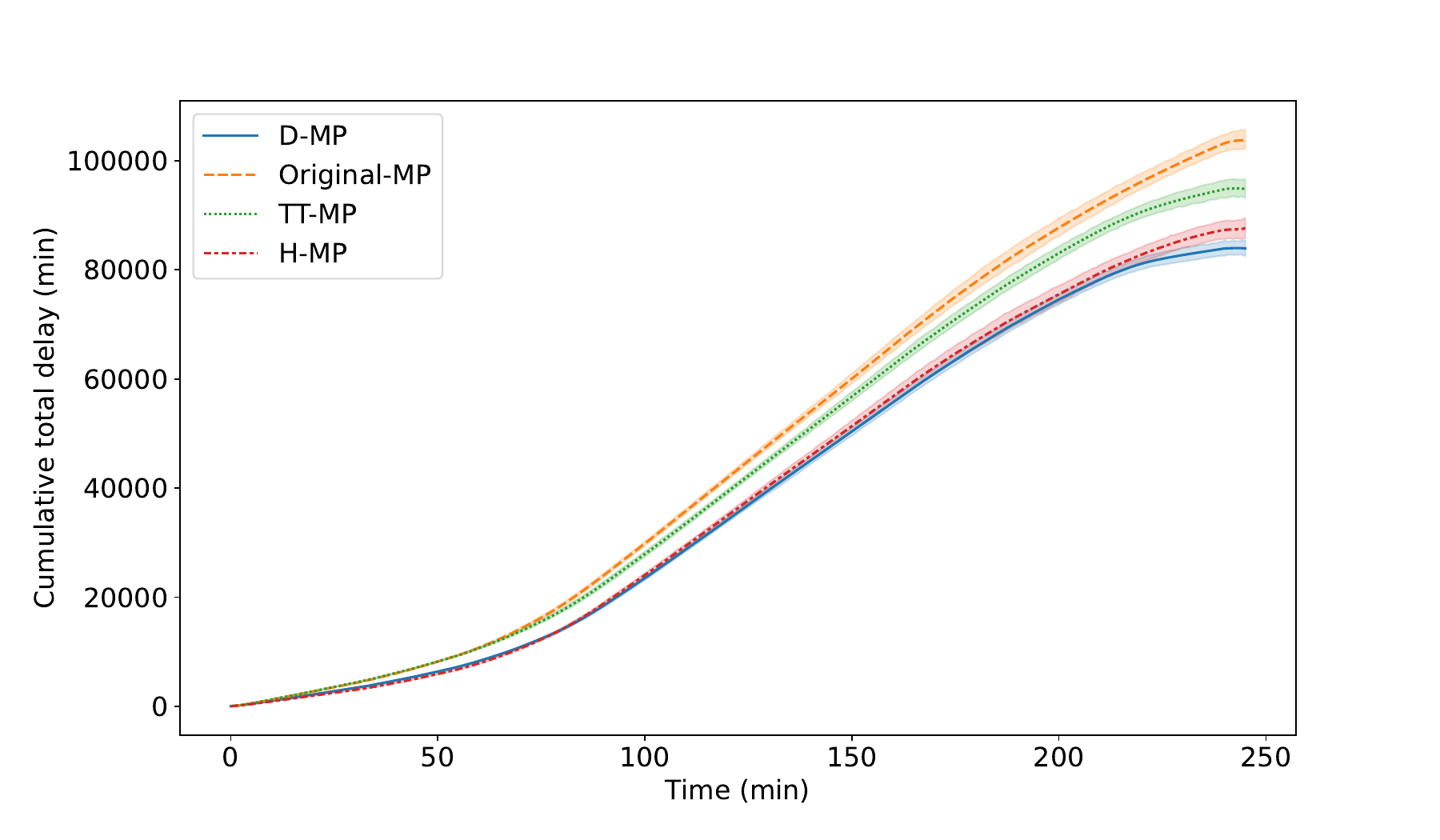}
         \caption{Cumulative delay}
         \label{fig:delay_cumulative}
     \end{subfigure}
     \caption{Comparison of internal delay.}
     \label{fig:comparison_delay}
\end{figure}
\begin{figure}[!htbp]
    \centering
    \includegraphics[width=3.5in]{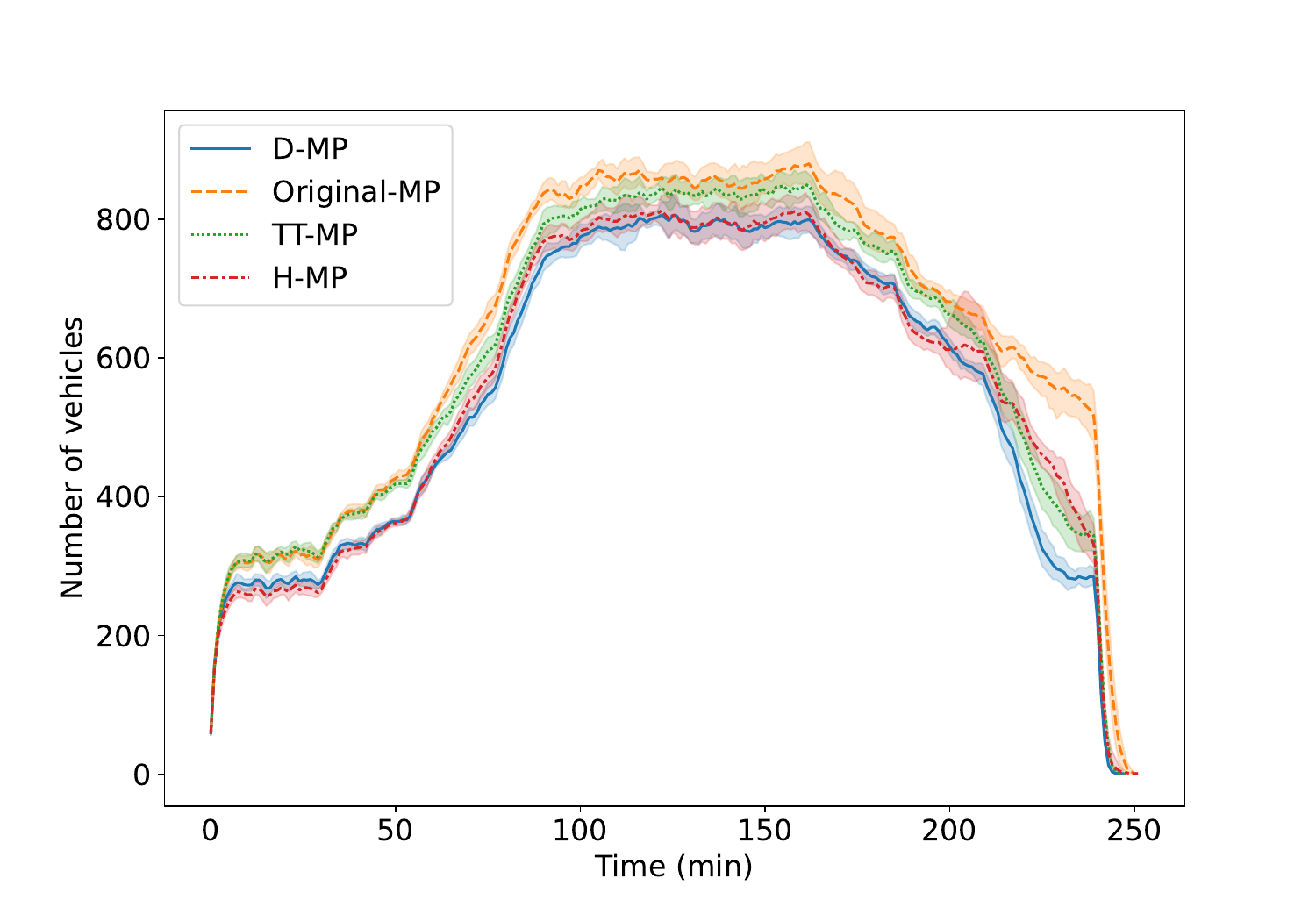}
    \caption{Number of vehicles in the network}
    \label{fig:internal_veh}
\end{figure}

\subsubsection{Blocked vehicles}
Figure \ref{fig:comparison_delay} shows only the delay incurred by the vehicles already in the network, but it does not take into consideration  vehicles blocked at the entry nodes. If a vehicle cannot be inserted at the predefined depart time due to queue spillover, it has to wait until there is enough space on the entry link. Does the lower density from D-MP imply fewer vehicles are able to enter the network? Figure \ref{fig:n_blocked} shows the comparison of the number of blocked vehicles. As mentioned earlier, we defined the stable region as the set of demand that can be accommodate by a model. Before $t=70$ min, demand is low and appears to be stable for all models. As demand increases, Original-MP is the first model that starts blocking vehicles from entering the network, which indicates it has the smallest stable region under this specific demand pattern. It also has the highest number of blocked vehicles during the whole simulation. On the contrary,  D-MP has the largest stable region and least number of blocked vehicles. The evolution of average waiting time, which is defined as the mean time all vehicles up to now (not just the blocked vehicles) have to wait for being inserted to the simulation, is shown in Figure \ref{fig:waiting}. Similarly, D-MP has the lowest waiting time while Original-MP has the highest.

\begin{figure}[!htbp]
     \centering
     \begin{subfigure}{0.45\textwidth}
         \centering
         \includegraphics[width=\textwidth]{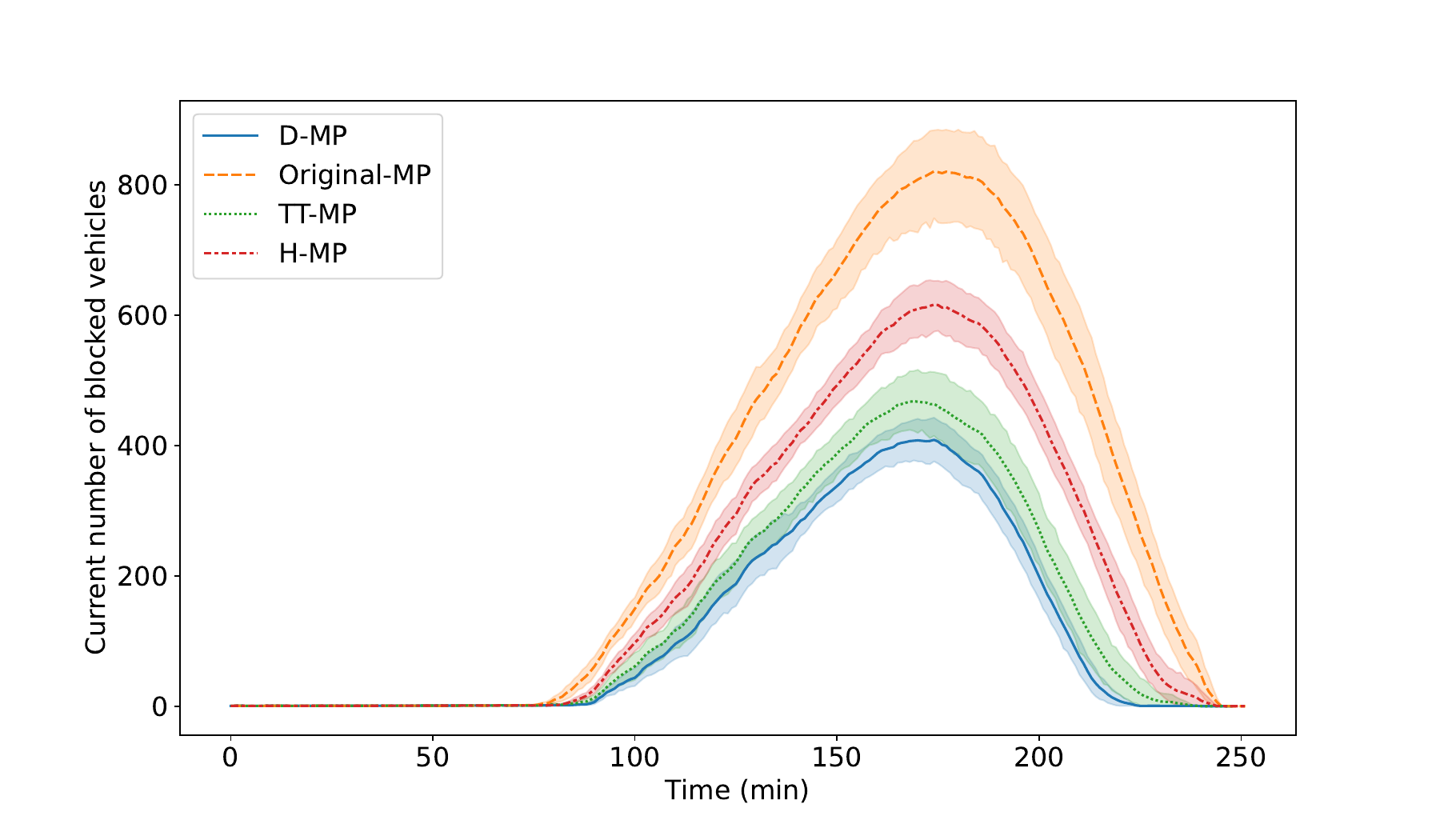}
         \caption{Number of blocked vehicles.}
         \label{fig:n_blocked}
     \end{subfigure}
     \begin{subfigure}{0.45\textwidth}
         \centering
         \includegraphics[width=\textwidth]{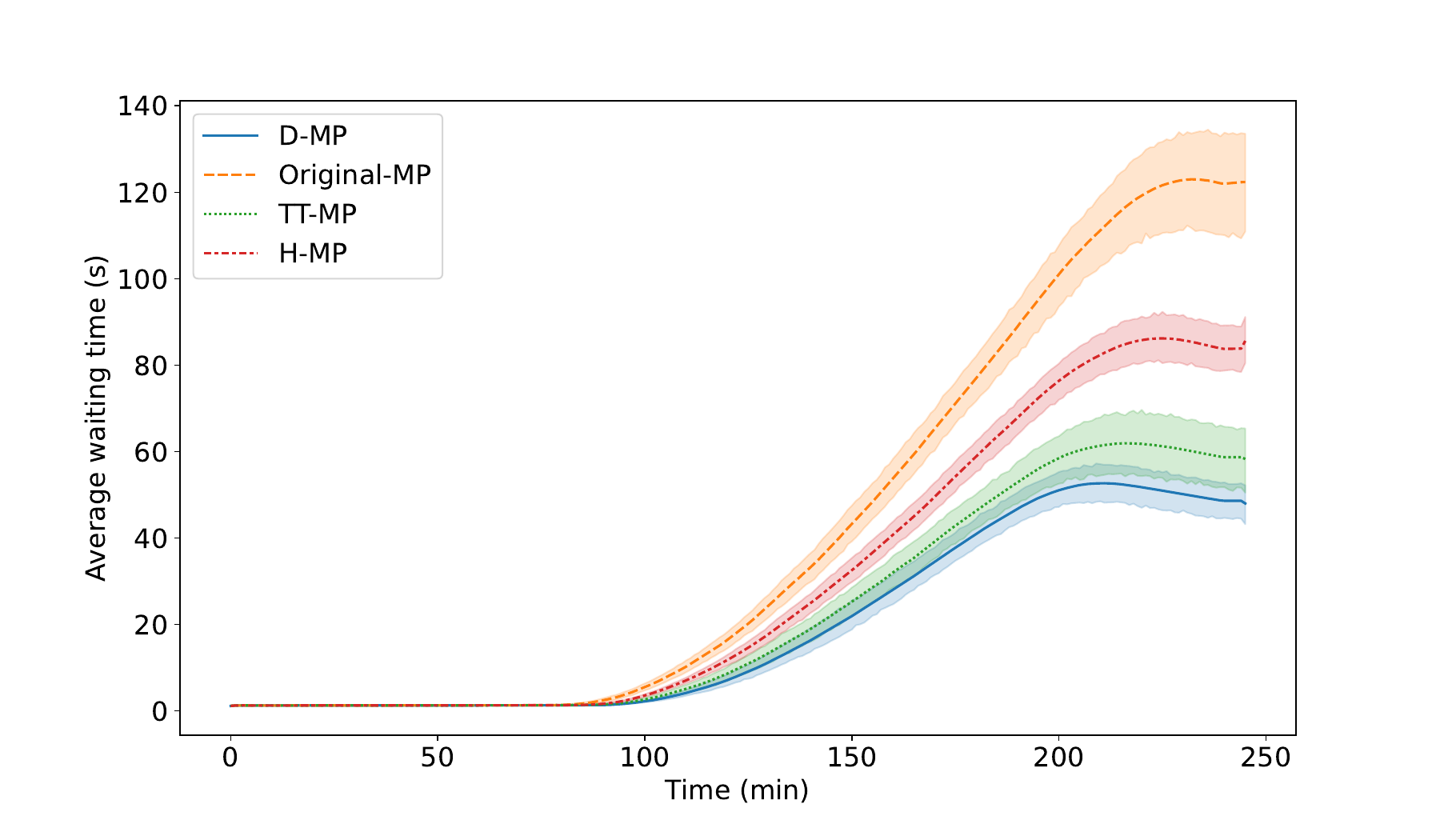}
         \caption{Waiting time.}
         \label{fig:waiting}
     \end{subfigure}
     \caption{Comparison of blocked vehicles.}
     \label{fig:comparison_blocked}
\end{figure}

Unlike the internal delay shown in Figure \ref{fig:comparison_delay}, TT-MP performs better than H-MP in this metric, i.e., TT-MP blocks fewer vehicles than H-MP. Also, the percent difference in blocked vehicles between D-MP and H-MP is much more significant than the internal delay. The reason is that it is more reasonable to have relatively long green time and reduce phase switches when demand is high so that the queues served by the activated phase can be cleared before the signal changes to the next phase. However, the instantaneous-based metric of the  H-MP tends to switch the signal more frequently than what is optimal. In SUMO, a halting vehicle is defined as a vehicle with speed lower than 0.1 m/s. After the phase serving the entry link is activated, all vehicles start accelerating and moving forward. In a very short time, these vehicles are not detected as halting vehicles as long as they speed up to 0.1 m/s, although their speed can still be low (and are practically stopped). In this case, H-MP changes phase too frequently so that more vehicles can be blocked. On the contrary, the average-based algorithm, D-MP, can overcome this drawback since it uses delay, which increases whenever the speed is lower than the speed limit. Therefore, if the activated phase has a low average speed and large number of vehicles at the end of the current step, it is likely to be activated for the next step as well in D-MP. Then, we calculated the average total delay, which is equal to the sum of average internal delay and average waiting time, for all models, as shown in Table \ref{tab:delay_compare}. Overall, the average-based models have lower average delay than the instantaneous-based models. Compared to the benchmark models, the proposed D-MP can reduce delay by 13.11\%-36.44\%.\\
\begin{table}[htbp]
  \centering
  \caption{Delay comparison (s)}
    \begin{tabular}{ccccc}
    \toprule
    Model & Internal & Waiting & Total & Total reduction (\%) \\
    \midrule
    D-MP  & 136.73 & 48.11 & 184.84 & - \\
    H-MP  & 142.32 & 83.33 & 225.65 & 18.08 \\
    Original-MP  & 168.96 & 121.88 & 290.84 & 36.44 \\
    TT-MP & 154.52 & 58.21 & 212.73 & 13.11 \\
    \bottomrule
    \end{tabular}%
  \label{tab:delay_compare}%
\end{table}%

\subsubsection{Inflow and throughput}
Figure \ref{fig:comparison_vehicle} shows the difference in the cumulative number of entered, exit vehicles through all enter and exit nodes between each model and the D-MP, which generally has the highest values. The vertical axis is the cumulative number of vehicles from a given model subtracted by the cumulative value obtained using the D-MP, so negative values indicate the cumulative inflows/outflows for the corresponding model are lower than D-MP. The reason why we show the difference instead of the cumulative number is that compared to the absolute cumulative number, the difference is very small, so it is hard to observe the difference in cumulative plots.  The results reveal  shows that both the entry flow and exit flow from D-MP are higher than all benchmark models, except for the H-MP under very low demand. Under this situation, H-MP has a slightly higher throughput than D-MP, but it becomes lower before the demand becomes unstable (around 80 mins), as shown in Figure \ref{fig:n_blocked}.

\begin{figure}[!htbp]
     \centering
     \begin{subfigure}{0.45\textwidth}
         \centering
         \includegraphics[width=\textwidth]{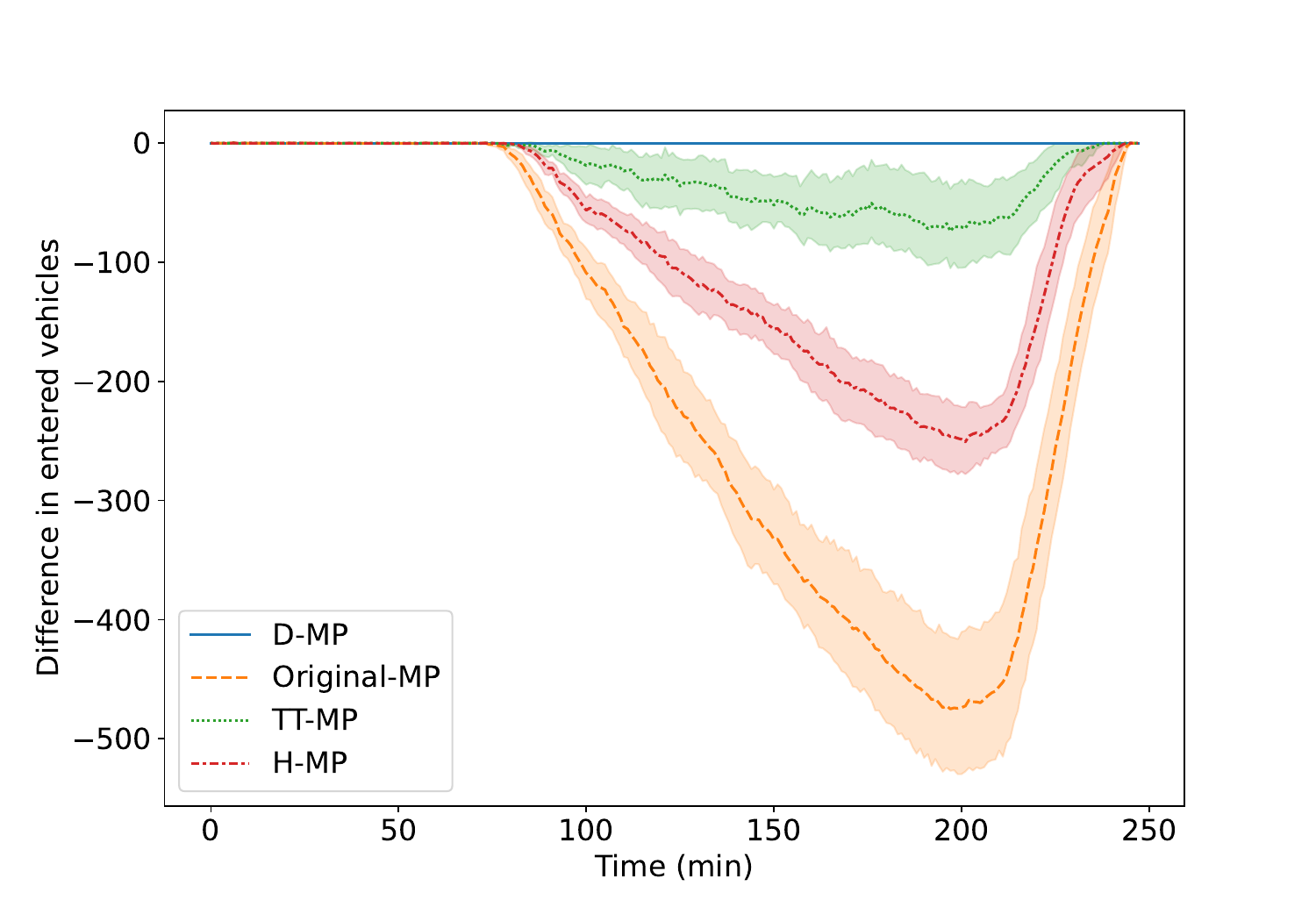}
         \caption{Entered vehicles.}
         \label{fig:diff_enter}
     \end{subfigure}
     \begin{subfigure}{0.45\textwidth}
         \centering
         \includegraphics[width=\textwidth]{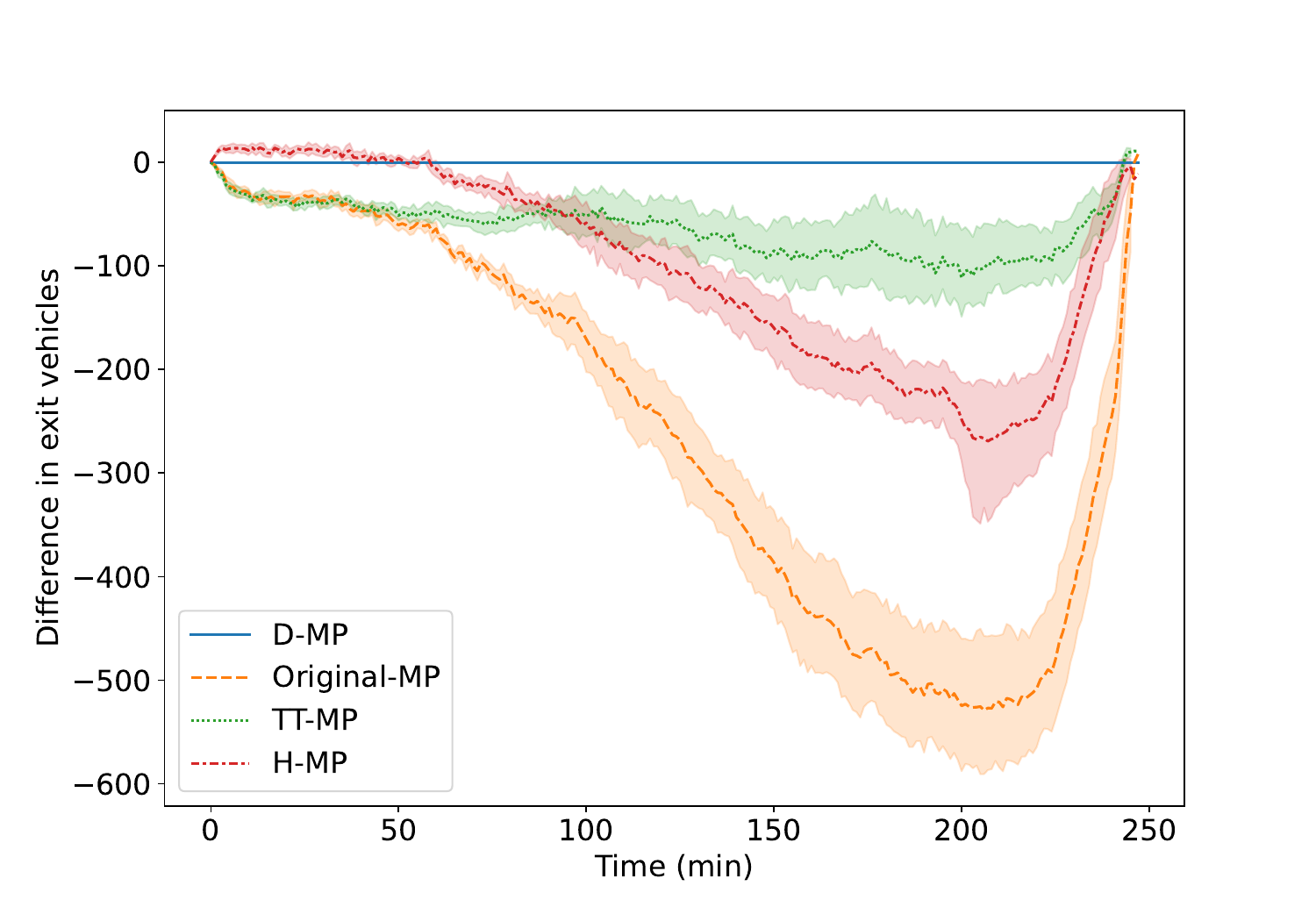}
         \caption{Exit vehicles.}
         \label{fig:diff_exit}
     \end{subfigure}
     \caption{Comparison of blocked vehicles.}
     \label{fig:comparison_vehicle}
\end{figure}

Overall, this section shows that under a realistic and typical traffic pattern, e.g., off-peak hours->peak hours->off-peak hours, the proposed D-MP outperforms all three benchmark models in the following ways. First, D-MP has the lowest density, internal delay and waiting time. Second, D-MP has the largest stable region, which implies when demand increases, D-MP accommodates more vehicles before there are vehicles blocked at the boundary than other models. Third, D-MP has the highest average inflow and throughput rate during the whole period. 

\subsection{Performance in a connected vehicle environment}
The results above are obtained based on full knowledge of the required metrics. However, in reality, the availability or accuracy of these metrics is subject to the availability of corresponding devices or estimation models. Normally, it is expensive or even impossible to equip all intersections with the measurement devices, which is a bottleneck for the implementation of these algorithms in the real world. Thanks to the development and increasing popularity of connected vehicles (CVs), we can expect that more types and more accurate real-time data can be obtained through the communication with CVs directly without the help of other equipment in the future. Therefore, we investigated the performance of the proposed model in a CV environment. Here, we assume a portion of the vehicles in the network are connected, that these CVs are randomly distributed in the network, and only the information from CVs can be obtained. Then, we used the estimation based on the CVs data to control the signal according to the models. We use the same network and same demand scenario in Section \ref{sec:timestep} and tested the performance of the models under various CV penetration rates.

\begin{figure}[!htbp]
     \centering
     \begin{subfigure}{0.45\textwidth}
         \centering
         \includegraphics[width=\textwidth]{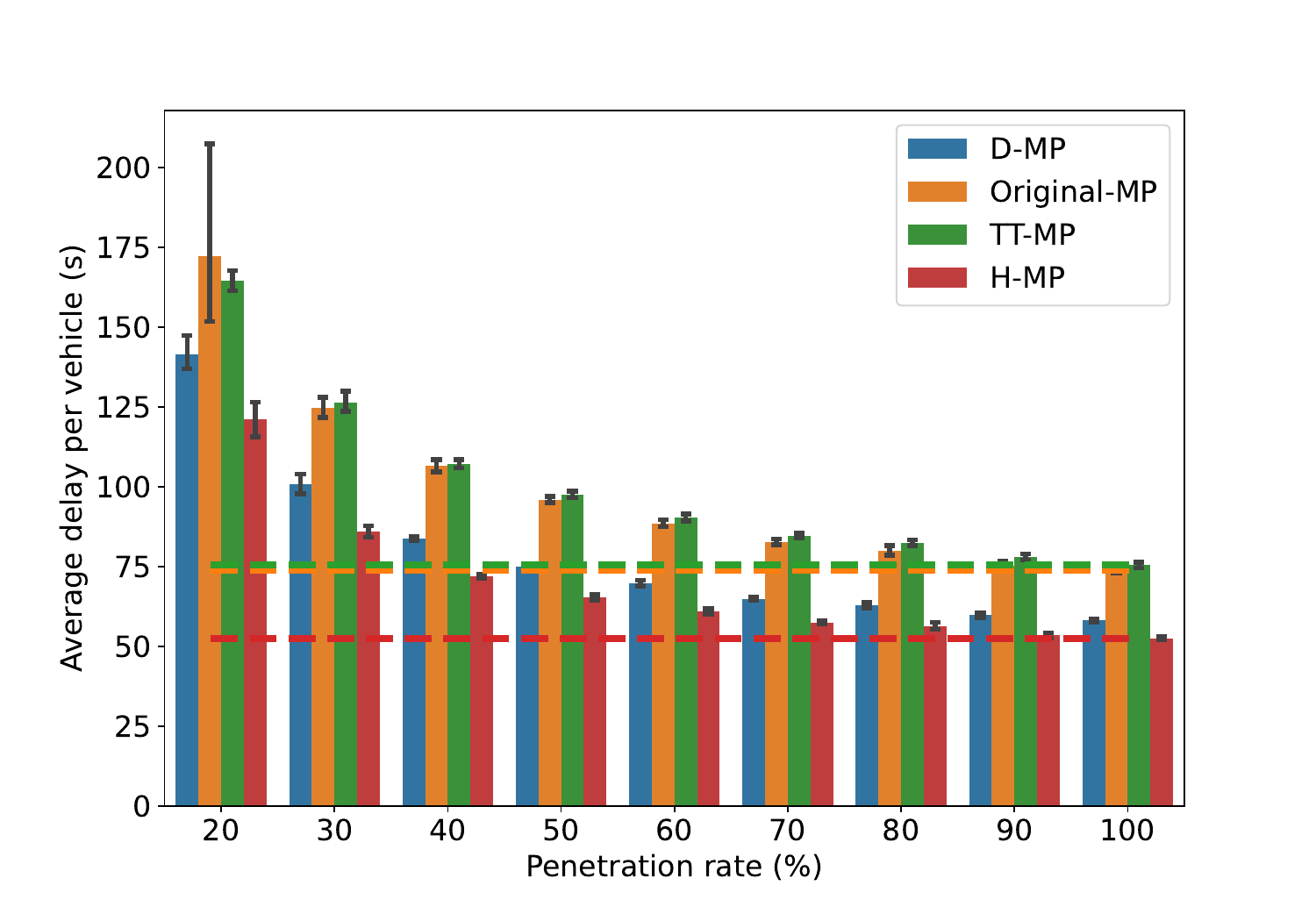}
         \caption{Low demand.}
         \label{fig:pr_low}
     \end{subfigure}
     \begin{subfigure}{0.45\textwidth}
         \centering
         \includegraphics[width=\textwidth]{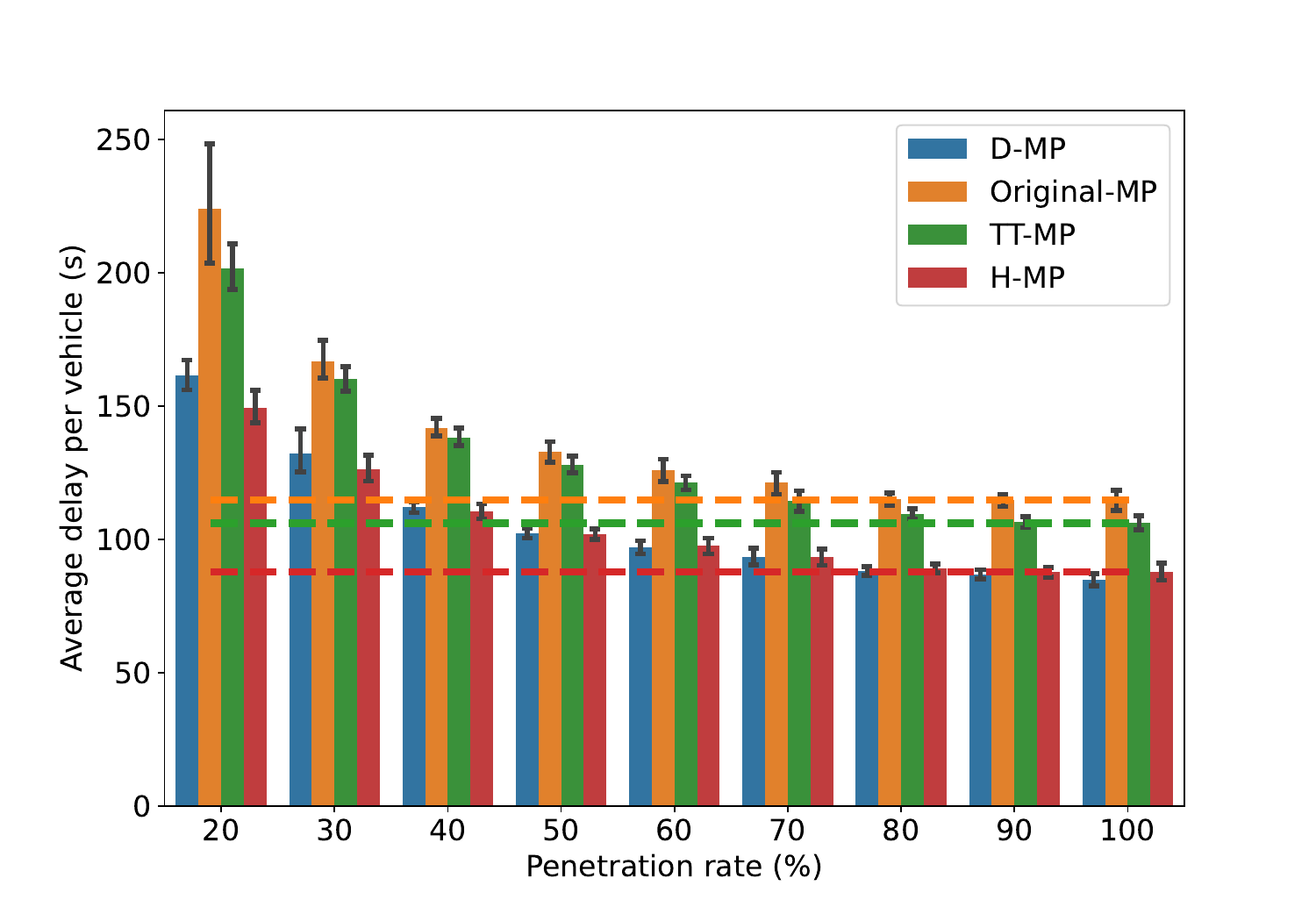}
         \caption{Medium demand.}
         \label{fig:pr_medium}
     \end{subfigure}
     \hfill
     \begin{subfigure}{0.45\textwidth}
         \centering
         \includegraphics[width=\textwidth]{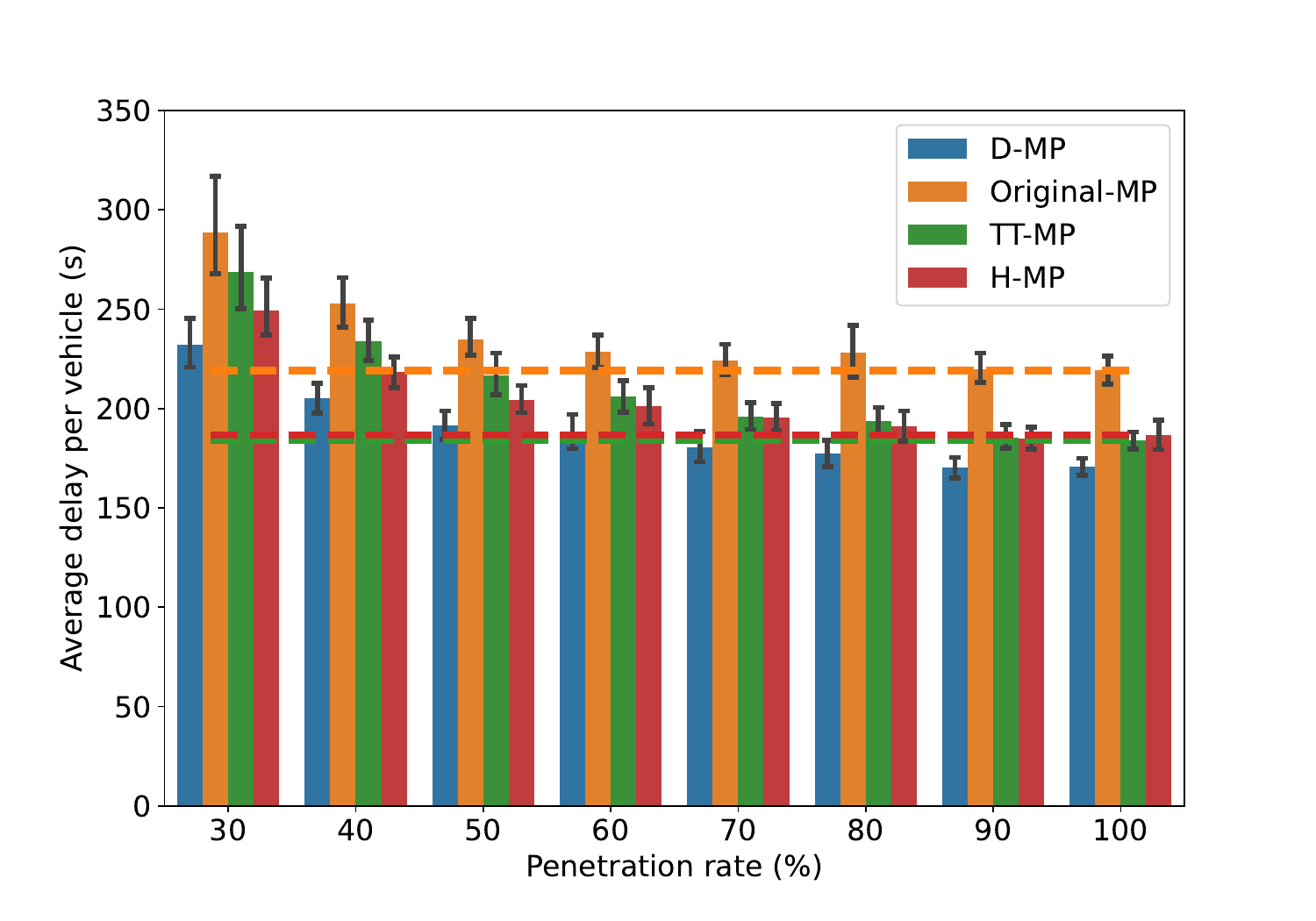}
         \caption{High demand.}
         \label{fig:pr_high}
     \end{subfigure}
     \caption{Effect of penetration rate on control performance.}
     \label{fig:pr}
\end{figure}

Figure \ref{fig:pr} \hl{shows the average delay per vehicle from D-MP and all benchmark models with optimal time steps for each under various penetration rates of connected vehicles. The results with penetration rate equal to 10\% are removed since both the average value and standard deviation are too large, which implies the control performance is poor under extremely low penetration rate. Note that the results with penetration rate equal to 20\% under high demand condition are also removed since the standard deviation from} Original-MP \hl{and TT-MP is too large. The dashed horizontal lines indicate the average delay from the benchmark models in a fully connected environment, i.e., penetration rate equals 100\%. As expected, both the average delay and variation of delay from all models decrease with the rise in CV penetration rate. In addition, this effect is more significant when the penetration rate is low. For example, the change in delay when penetration rate rises from 20\% to 30\% is much larger than when it rises from 90\% to 100\%. The results reveal that under low demand conditions, D-MP outperforms both} Original-MP \hl{and TT-MP for all penetration rates. When the penetration rate exceeds 50\%, D-MP prevails over} Original-MP \hl{and H-MP in a fully connected environment. H-MP has lower delay than D-MP with all penetration rates under low demand conditions. However, D-MP eventually outperforms H-MP as demand increases. Under medium demand conditions, D-MP has lower delay when the penetration rate exceeds 60\%, and the delay from D-MP when the penetration rate exceeds 80\% is lower than H-MP in a fully connected environment. Under high demand conditions, this threshold decreases to 70\%. For simplicity, we omitted the comparison of other metrics because they lead to a similar conclusion.}


\section{Conclusion}\label{sc:cc}
This paper proposes a novel MP network-wide signal control algorithm, which uses the total travel delay in the previous time step to calculate weight for movements. \hl{With the development and ubiquitous of probing vehicles or mobile devices, real-time travel delay can be readily measured at a low cost, which makes the proposed model feasible in practical.} Under the assumption of infinite queue capacity, the model is analytically proven to maximize the network throughput, which is the key property of the Original-MP \citep{varaiya2013max}. The performance of the proposed variant is compared with other MP frameworks using the microscopic traffic simulation in SUMO.  First, we showed that time step for decision-making  plays a key role on the model performance and selected the optimal time steps for each model according to the simulation results. Then, we applied all the models to a typical time-varying traffic condition. \hl{Overall, the proposed model outperforms three benchmark models:} Original-MP, \hl{TT-MP and H-MP which use the number of vehicles on a link, total travel time and number of stopped (halting) vehicles, respectively, as the metric to define weight and pressure,} in average delay, blocked vehicles and network throughput, although H-MP has a slightly better performance than the proposed model in low traffic conditions. One benefit of the proposed delay-based model is that it can be implemented using measurements taken from a subset of vehicles. Thus, we explored the results if the required metrics for the models are not fully known by assuming a CV environment. The control performance of the proposed model was tested under different penetration rates. As expected, a higher penetration rate leads to a better control performance. We also studied the threshold of the penetration rate at which the proposed model is better than the benchmark models under a full CV environment. The threshold values are shown to be decreasing with the increasing in demand.

One shortcoming of the proposed model is that it activates phases in arbitrary orders with arbitrary durations (which have to the multiples of time step size). As is pointed out by \citep{levin2020max}, this arbitrary phase selection can be confusing to drivers and thus unacceptable for some city engineers. \citep{levin2020max} addressed this issue, to some extent, by restricting the cycle length to be upper bounded by a pre-defined value and fixing the phase sequence. However, since the cycle length changes from cycle to cycle and from intersection to intersection, the signals are not coordinated. To this end, it is meaningful to develop a coordination strategy combining with the proposed D-MP to improve its practicability. In addition, this model only tackles the signal control inside a network. To develop a framework that combines the perimeter control at the boundary of a network and the proposed model for the inside signals is another interesting direction. 
\section*{Acknowledgement}
This research was supported by NSF Grant CMMI-1749200.
\begin{appendix}
\section{Time step selection for the benchmark models}\label{app:ts}
Figure \ref{fig:timestep_QMP} to Figure \ref{fig:timestep_HMP} show the influence of time step size on three metrics: average delay per vehicle, average queue length per link and average throughput, of three benchmark models. Based on the results, we use 9s, 9s and 5s as the optimal time steps for Original-MP, TT-MP and H-MP, respectively.
\begin{figure}[!htbp]
     \centering
     \begin{subfigure}{0.32\textwidth}
         \centering
         \includegraphics[width=\textwidth]{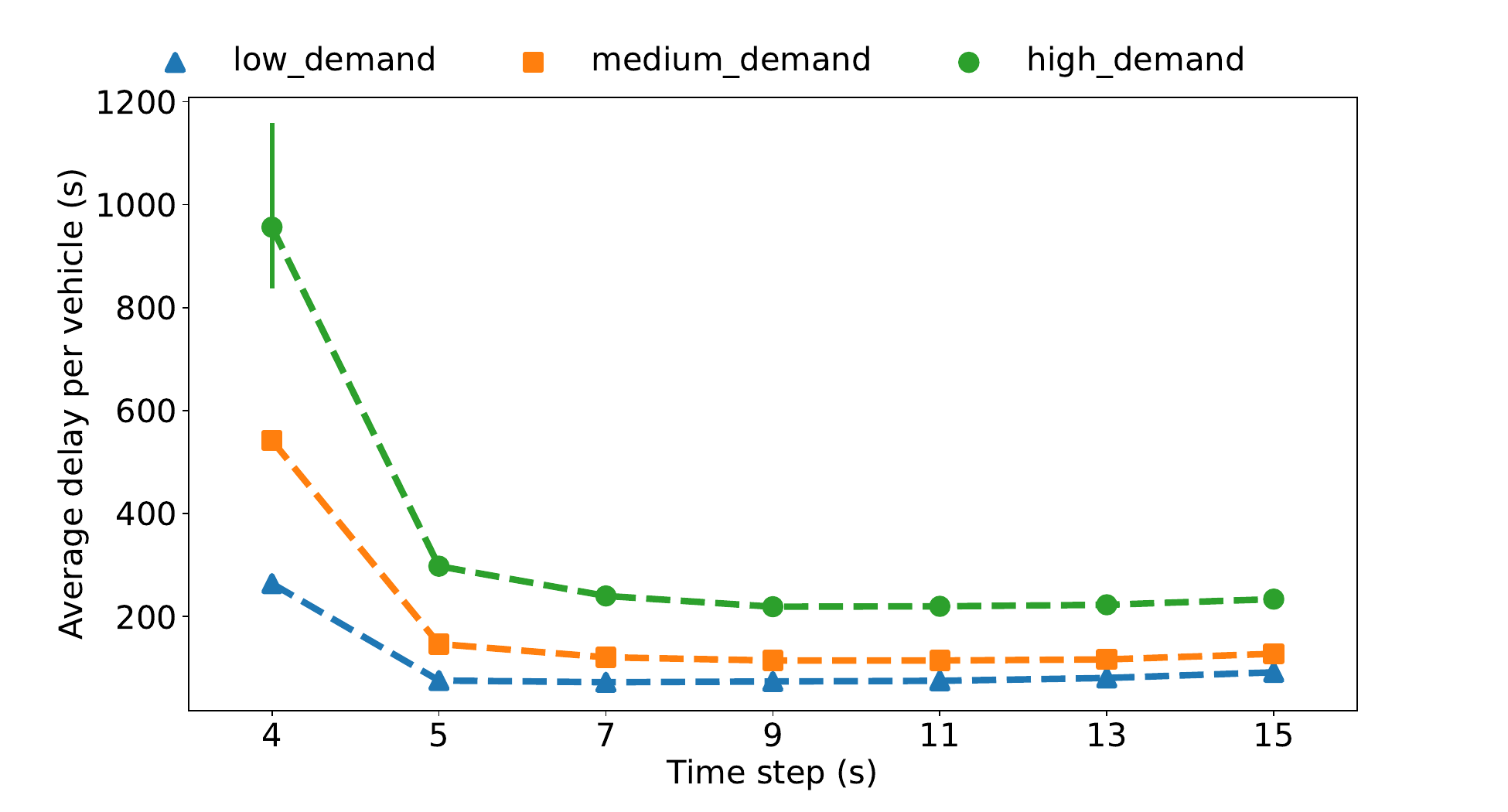}
         \caption{Average delay.}
     \end{subfigure}
     \begin{subfigure}{0.32\textwidth}
         \centering
         \includegraphics[width=\textwidth]{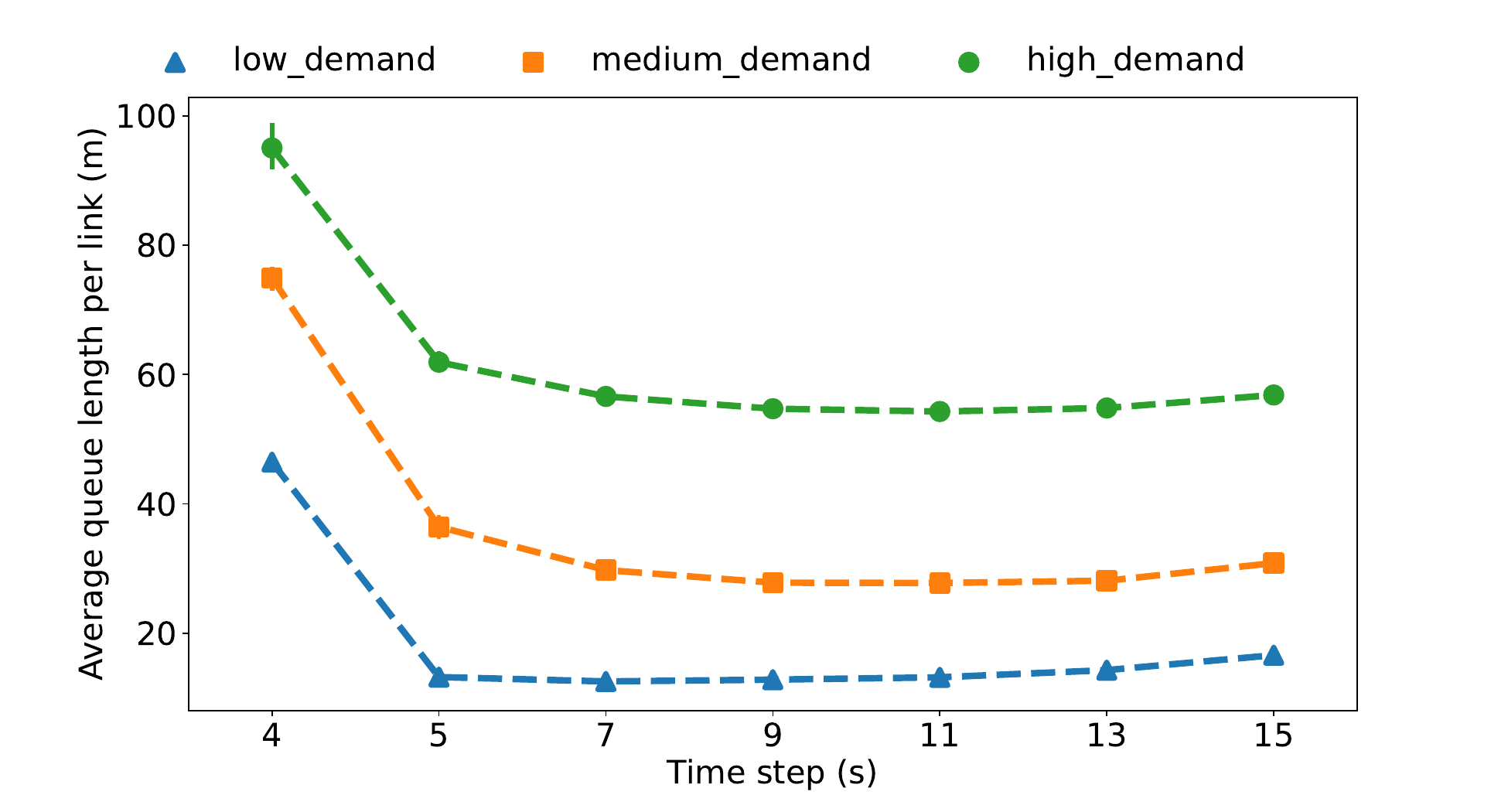}
         \caption{Average queue.}
     \end{subfigure}
     \begin{subfigure}{0.32\textwidth}
         \centering
         \includegraphics[width=\textwidth]{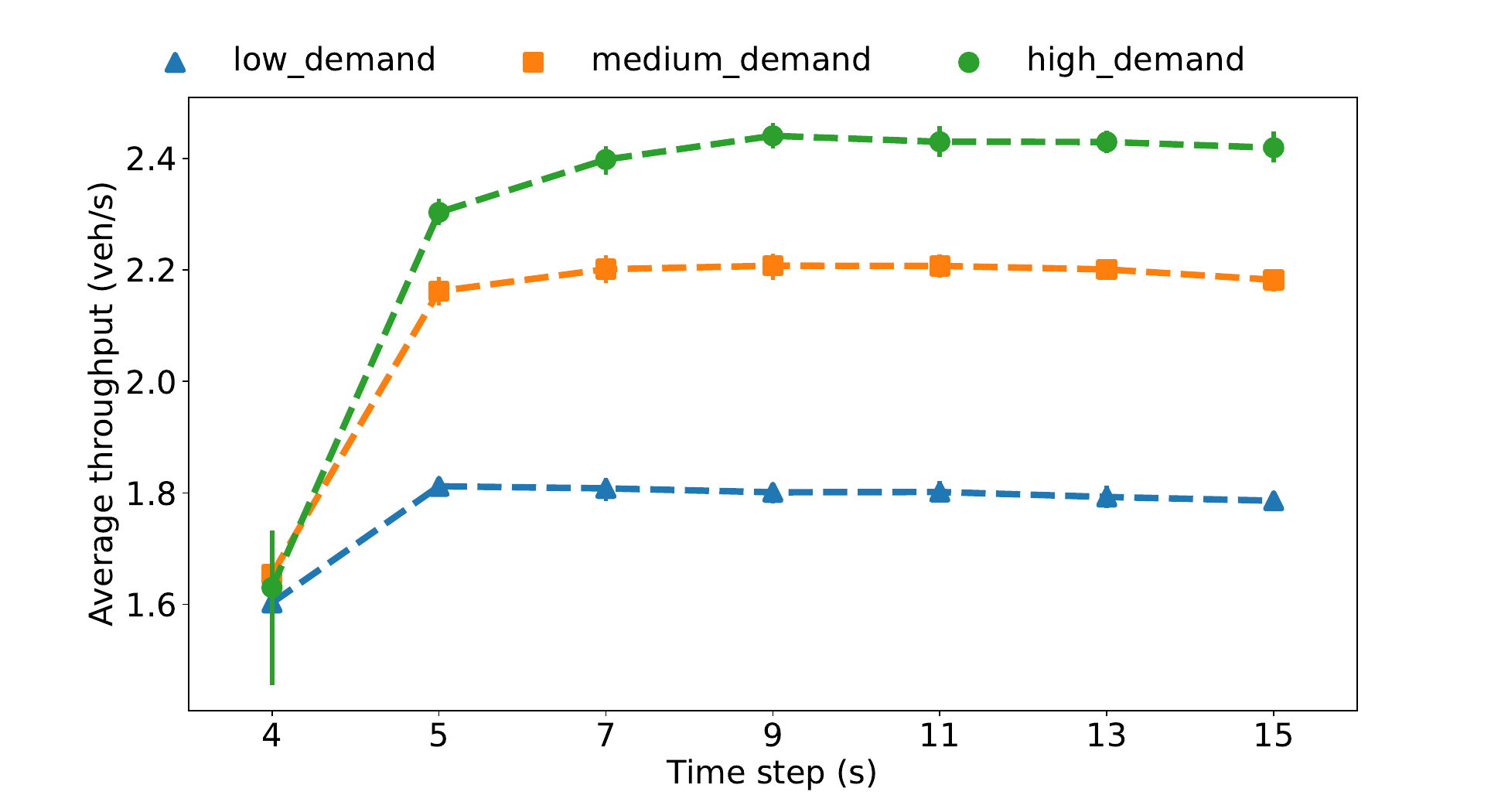}
         \caption{Average throughput.}
     \end{subfigure}
     \caption{Influence of time steps on Original-MP.}
     \label{fig:timestep_QMP}
\end{figure}
\begin{figure}[!htbp]
     \centering
     \begin{subfigure}{0.32\textwidth}
         \centering
         \includegraphics[width=\textwidth]{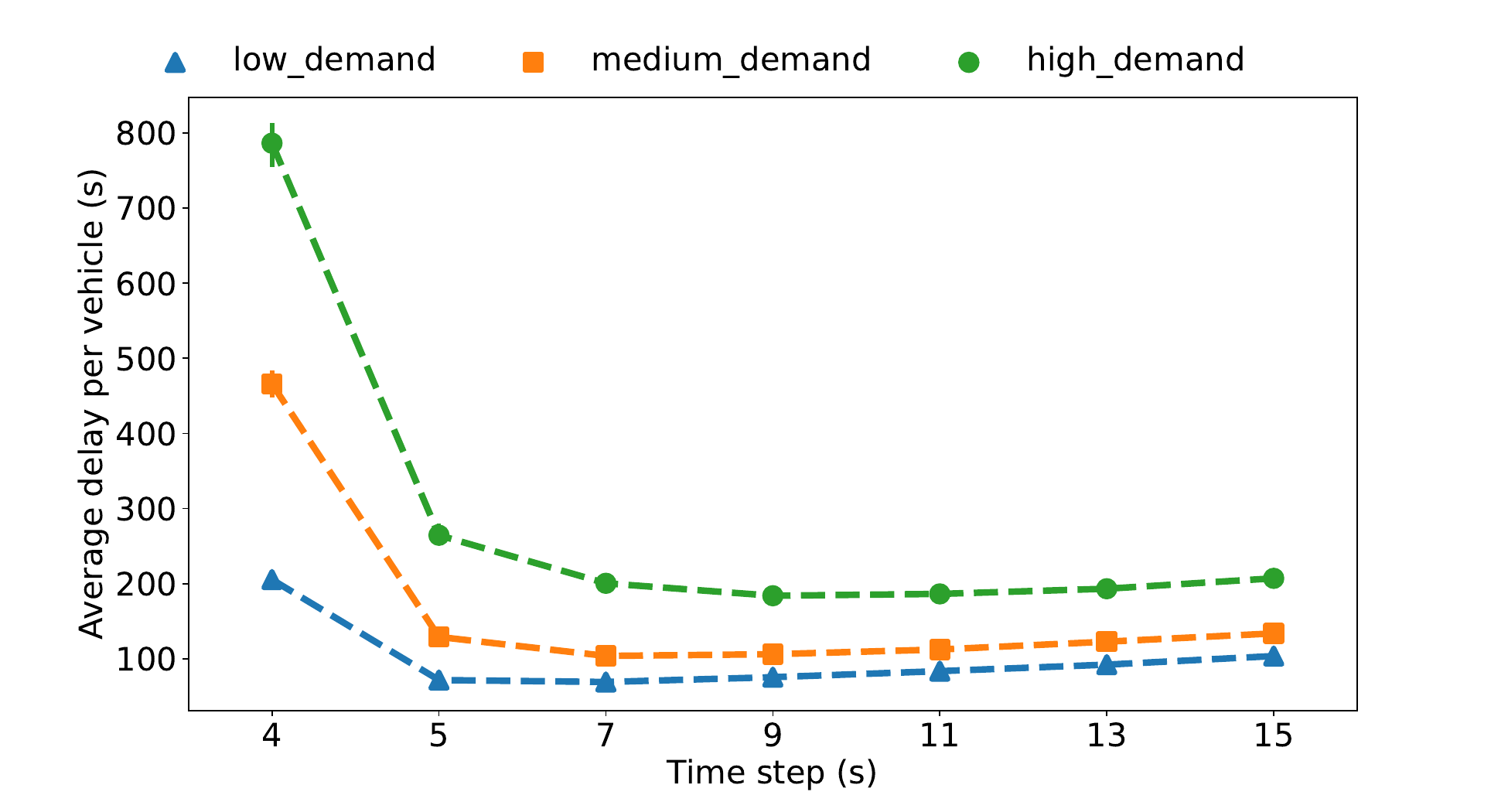}
         \caption{Average delay.}
     \end{subfigure}
     \begin{subfigure}{0.32\textwidth}
         \centering
         \includegraphics[width=\textwidth]{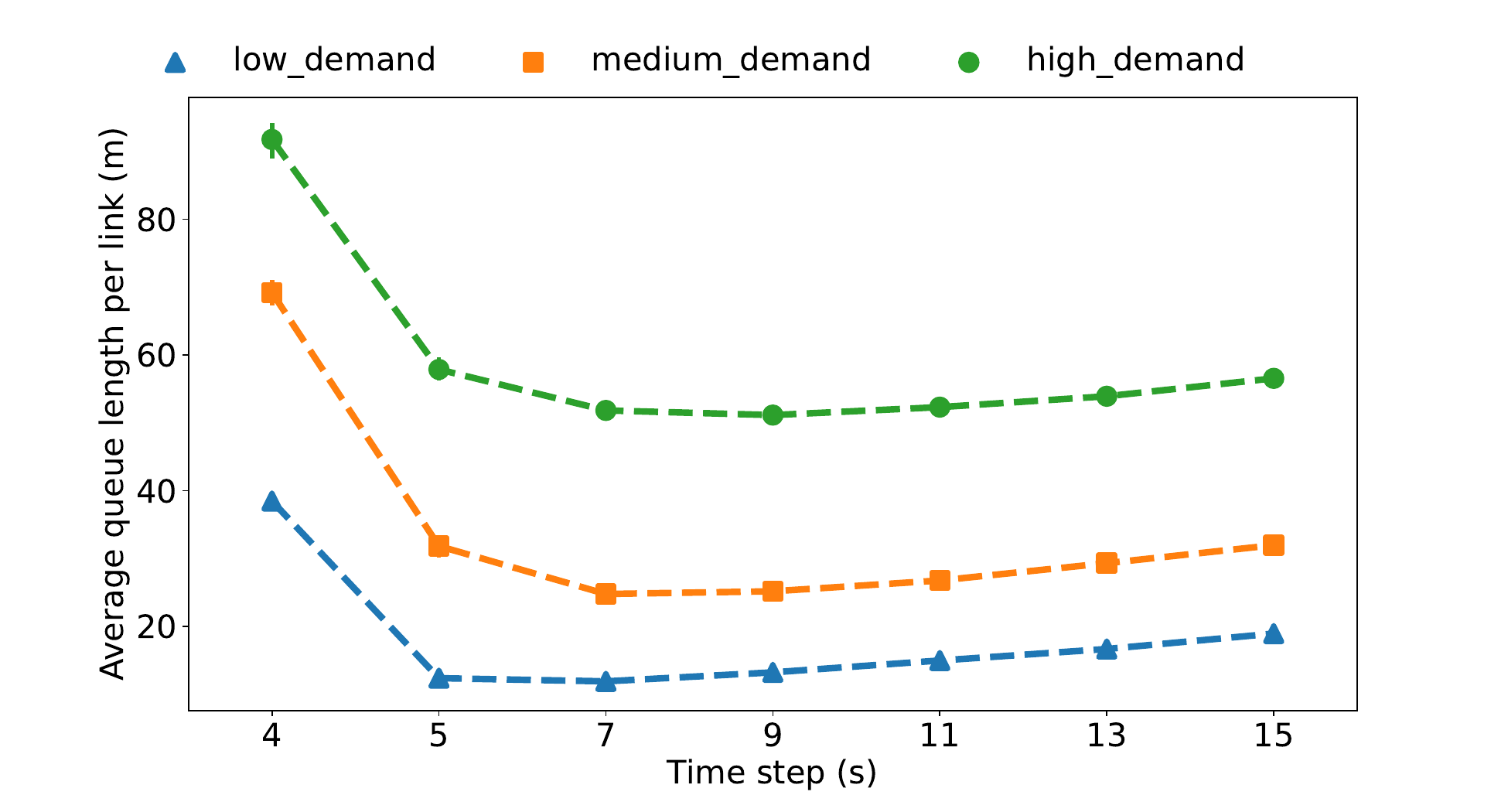}
         \caption{Average queue.}
     \end{subfigure}
     \begin{subfigure}{0.32\textwidth}
         \centering
         \includegraphics[width=\textwidth]{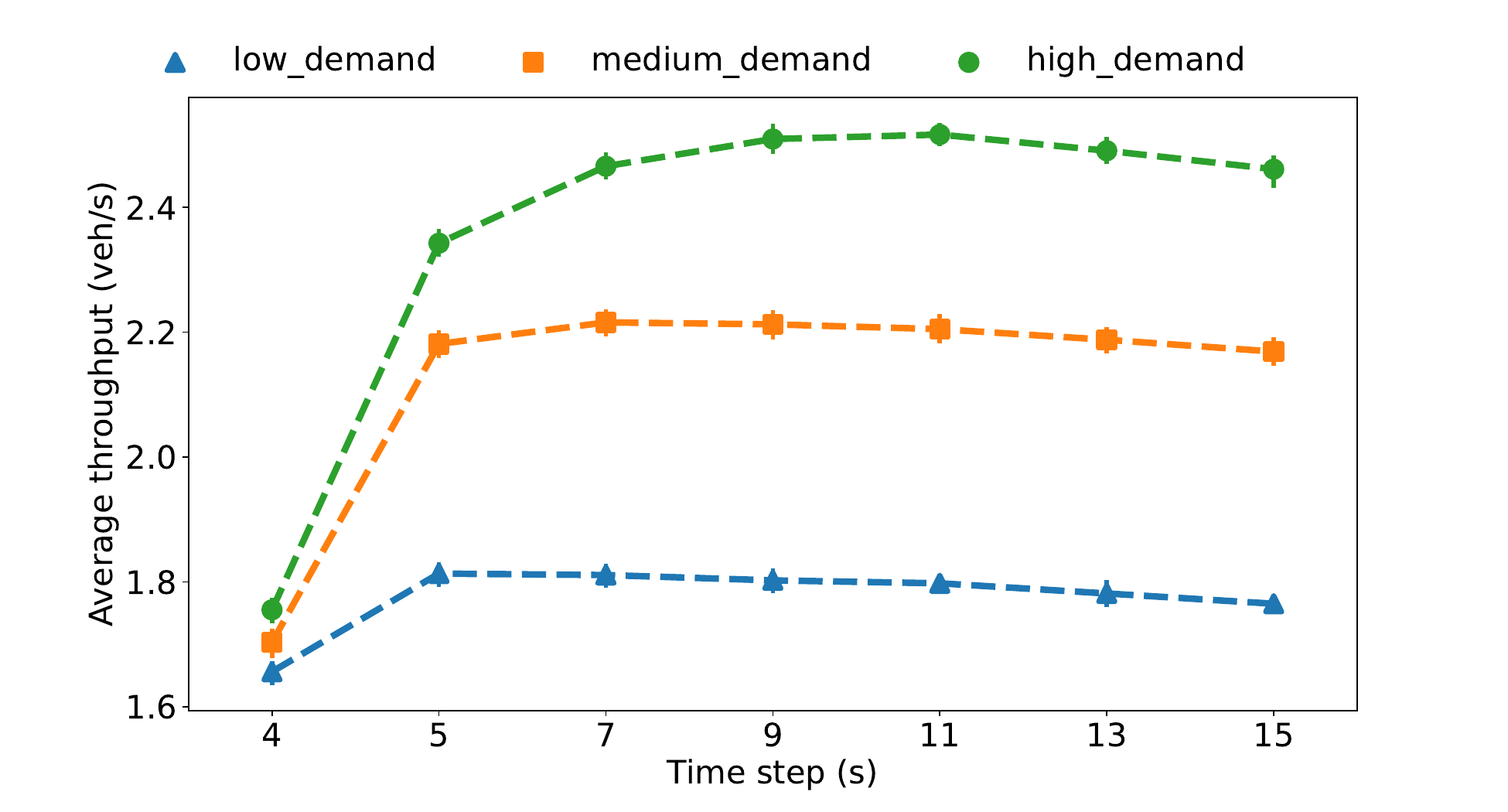}
         \caption{Average throughput.}
     \end{subfigure}
     \caption{Influence of time steps on TT-MP.}
     \label{fig:timestep_TTMP}
\end{figure}
\begin{figure}[!htbp]
     \centering
     \begin{subfigure}{0.32\textwidth}
         \centering
         \includegraphics[width=\textwidth]{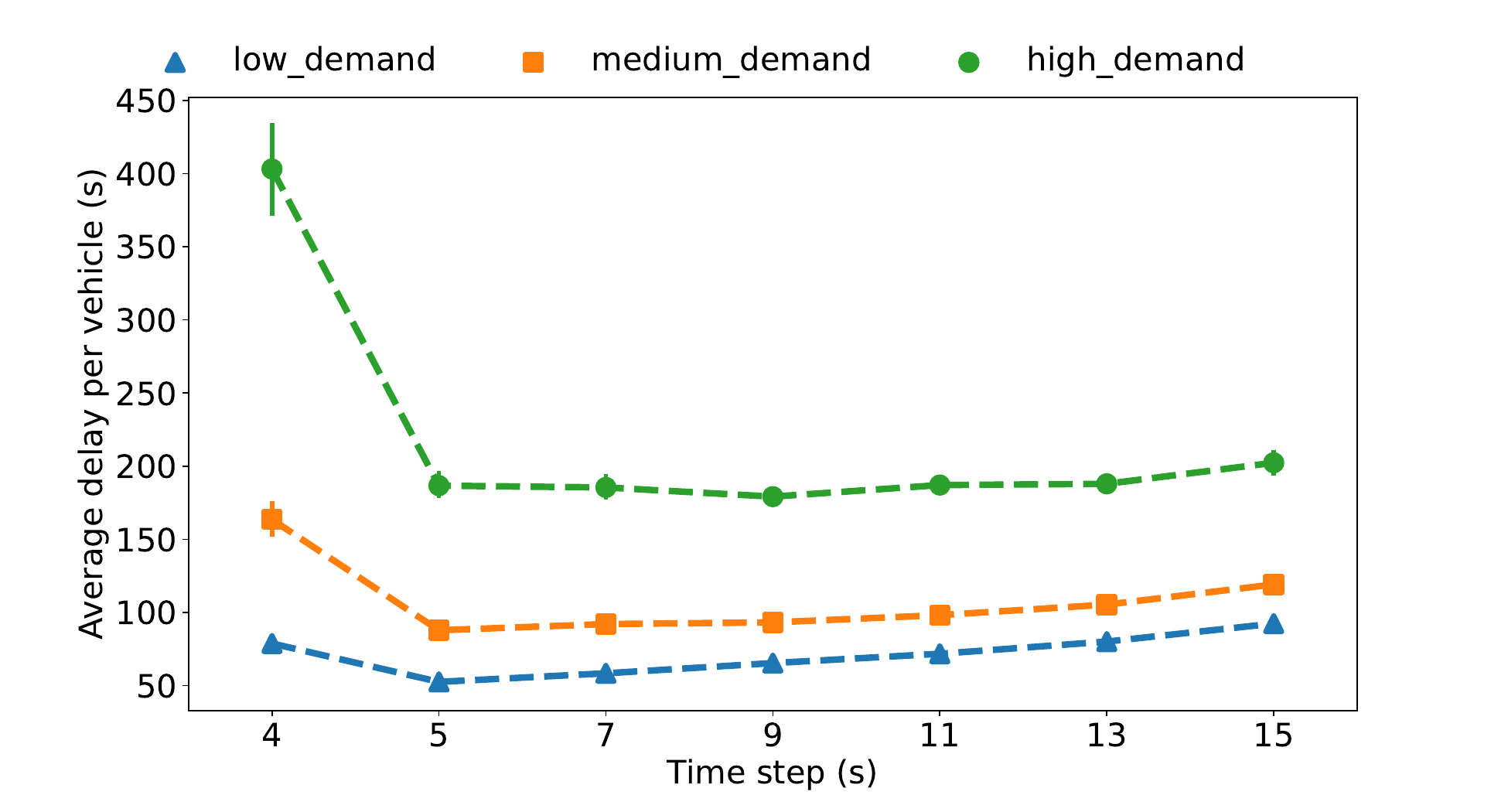}
         \caption{Average delay.}
     \end{subfigure}
     \begin{subfigure}{0.32\textwidth}
         \centering
         \includegraphics[width=\textwidth]{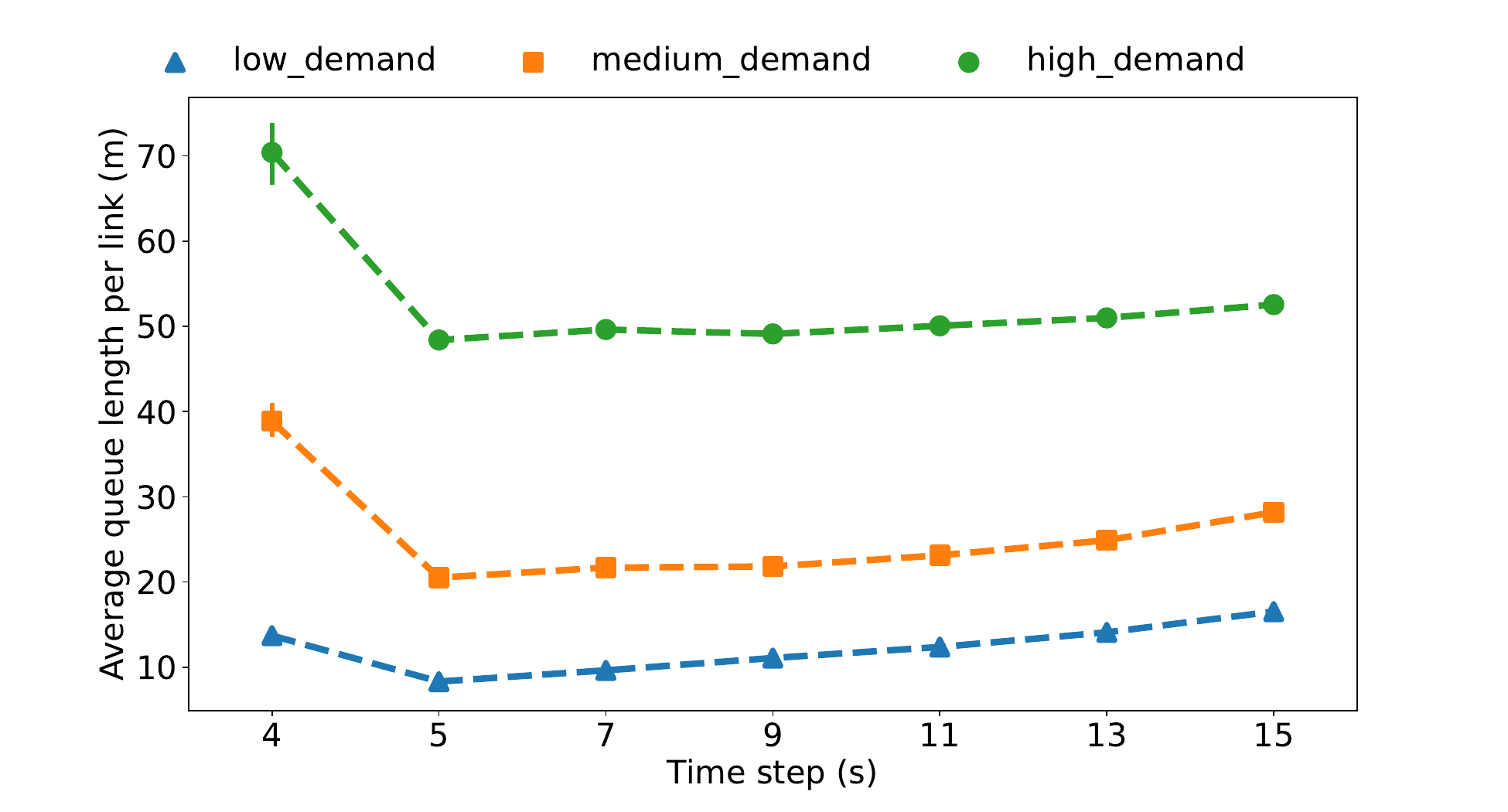}
         \caption{Average queue.}
     \end{subfigure}
     \begin{subfigure}{0.32\textwidth}
         \centering
         \includegraphics[width=\textwidth]{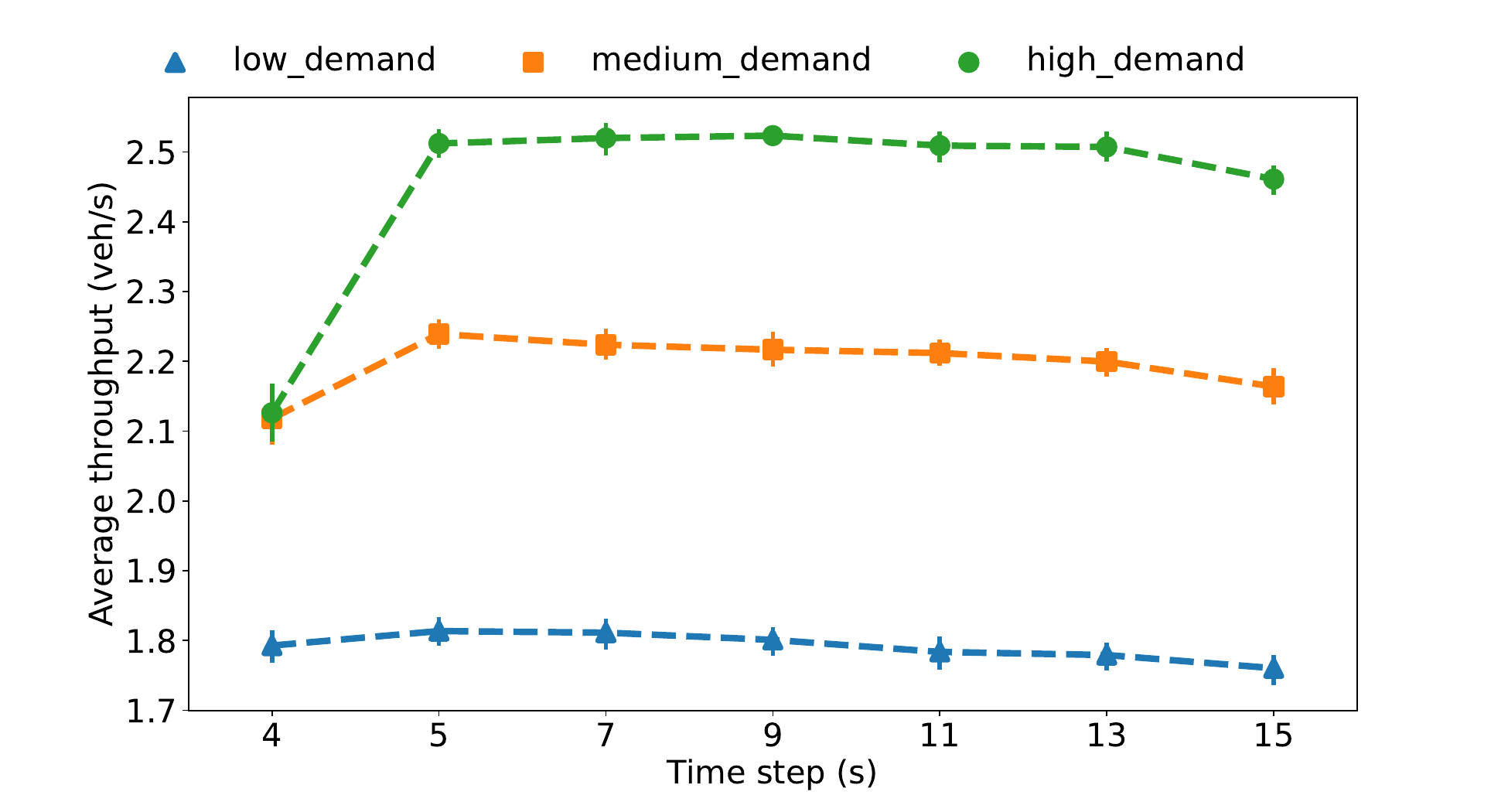}
         \caption{Average throughput.}
     \end{subfigure}
     \caption{Influence of time steps on H-MP.}
     \label{fig:timestep_HMP}
\end{figure}
\end{appendix}
\FloatBarrier
\bibliographystyle{cas-model2-names}

\bibliography{revised_2ndround}





\end{document}